\documentclass[11pt]{amsart}
\usepackage[utf8]{inputenc}
\usepackage{moresize}
\usepackage{blindtext, tikz-cd, enumerate}
\usepackage{subfig}
\usepackage{ytableau}
\usepackage{enumitem}
\usepackage{multicol}

\usepackage[hmargin=1.25in]{geometry} 
\usepackage{amssymb}
\usepackage{hyperref}
\hypersetup{colorlinks=true, linkcolor=red, filecolor=magenta, urlcolor=cyan, citecolor=magenta}
\usepackage{amsthm}
\usepackage{amsmath}
\usepackage{mathtools}
\theoremstyle{plain}
\newtheorem{thm}{Theorem}
\newtheorem*{thm1}{Theorem A}
\newtheorem*{thm2}{Theorem B}
\newtheorem*{prop2}{Proposition C}
\newtheorem*{cor3}{Corollary D}
\newtheorem{lemma}[thm]{Lemma}
\theoremstyle{definition}
\newtheorem{defn}[thm]{Definition}
\newtheorem{prop}[thm]{Proposition}
\newtheorem{cor}[thm]{Corollary}
\newtheorem{rem}[thm]{Remark}
\newtheorem{alg}[thm]{Algorithm}
\newtheorem{ex}[thm]{Example}

\numberwithin{equation}{section}
\numberwithin{thm}{section}

\newcommand{\dom}{\ensuremath{  \trianglerighteq}}

\DeclareMathOperator{\charge}{charge}
\DeclareMathOperator{\cocharge}{cocharge}

\DeclareMathOperator{\syt}{SYT}
\DeclareMathOperator{\ssyt}{SSYT}
\DeclareMathOperator{\ctype}{ctype}
\DeclareMathOperator{\Frob}{Frob}
\DeclareMathOperator{\Hilb}{Hilb}
\DeclareMathOperator{\sh}{Sh}
\DeclareMathOperator{\shape}{shape}
\DeclareMathOperator{\rev}{rev}
\DeclareMathOperator{\cc}{cc}
\DeclareMathOperator{\cw}{c}

\DeclareMathOperator{\des}{Des}

\DeclareMathOperator{\sgn}{sgn}
\DeclareMathOperator{\rw}{rw}
\def\Frobq{\Frob_q}
\def\Hilbq{\Hilb_q}
\def\mHL{\tilde{H}}

\title{A charge monomial basis of the Garsia-Procesi ring }
\author{Mitsuki Hanada}
\address{\scriptsize{Department of Mathematics, University of California, Berkeley
}}
\email{\scriptsize{mhanada@berkeley.edu}}

\begin{document}
\begin{abstract}
We construct a basis of the Garsia-Procesi ring using the catabolizability type of standard Young tableaux and the charge statistic.  
This basis turns out to be equal to the descent basis defined in \cite{CC}.
Our new construction connects the combinatorics of the basis with the well-known combinatorial formula for the modified Hall-Littlewood polynomials $\mHL_\mu[X;q]$, due to Lascoux, which expresses the polynomials as a sum over standard tableaux that satisfy a catabolizability condition. 
In addition, we prove that identifying a basis for the antisymmetric part of $R_{\mu}$ with respect to a Young subgroup $S_\gamma$ is equivalent to finding pairs of standard tableaux that satisfy conditions regarding catabolizability and descents. This gives an elementary proof of the fact that the graded Frobenius character of $R_{\mu}$ is given by the catabolizability formula for $\mHL_\mu[X;q]$.
\end{abstract}
\maketitle
\section{introduction}
The polynomial ring $\mathbb{C}[\mathbf{x}] = \mathbb{C}[x_1,\dots, x_n]$ has a natural $S_n$ action permuting the variables. 
Taking the quotient of this polynomial ring by the ideal generated by $S_n$-invariant polynomials with no constant term defines the classical \textit{coinvariant ring}:
\[R_{1^n}= \mathbb{C}[\mathbf{x}]/\langle e_1(\mathbf{x}),\dots,  e_n(\mathbf{x})\rangle,\] where $e_d(\mathbf{x})$ is the $d$th elementary symmetric function. 
The coinvariant ring $R_{1^n}$ is a graded version of the regular representation of $S_n$.
As a graded algebra with an $S_n$ action, it is isomorphic to the cohomology ring of the complex flag variety $\mathcal{F}_n$. 
There are two well-known monomial bases of $R_{1^n}$, the \textit{Artin basis} and the \textit{Garsia-Stanton descent basis}, both with elements indexed by permutations in $S_n$. 
In this paper, we will focus on subsets of the latter basis. 
  
Certain quotients $R_\mu$ of $R_{1^n}$ correspond to $\mathbb{C}S_n$-modules that come from Springer fibers. 
Given a partition $\mu$ of $n$, the Springer fiber $\mathcal{F}_\mu\subset \mathcal{F}_n$  is the set of flags that are fixed by a unipotent matrix with Jordan type $\mu$. The cohomology ring $H^{*}(\mathcal{F}_\mu)$ of the Springer fiber has an $S_n$ action which is compatible with that on $H^{*}(\mathcal{F}_n)$ \cite{Springer1978}. 

Though the definitions above can be extended to different Lie types, there are certain nice properties when we consider Type A. In particular, the map $H^*(\mathcal{F}_n) \to H^*(\mathcal{F}_\mu)$ is a surjection. 
From this, we have a concrete realization of $H^*(\mathcal{F}_\mu)$ as a quotient of $R_{1^n}$ by an $S_n$ invariant ideal. The following presentation is due to De Concini-Procesi \cite{deCon-Proc} and Tanisaki \cite{Tanisaki}:
\[R_\mu = \mathbb{C}[\mathbf{x}]/\langle e_d(S) \ | \ S\subset \mathbf{x}, d> |S| -p_{n-|S|}(\mu) \rangle\]
where $p_{k}(\mu)$ is the number of boxes of the Young diagram of $\mu$ that are not in the first $k$ columns.
Garsia and Procesi \cite{GP} constructed a monomial basis of $R_\mu$ that is a subset of the Artin monomials.
In light of their work, we refer to $R_\mu$ as the \textit{Garsia-Procesi ring} indexed by $\mu$. 
The question of finding a subset of the Garsia-Stanton descent monomials that is a monomial basis of $R_\mu$ was unresolved until recently, when Carlsson and Chou gave a construction using shuffles of descent compositions  \cite{CC}. 

The graded Frobenius character (see Section \ref{subsec: symm functions} for definition) of the Garsia-Procesi ring $R_\mu$ is given by the modified Hall-Littlewood polynomials 
$\tilde{H}_\mu[X;q]$ \cite{HottaSpringer}. These polynomials have a combinatorial formula due to  
Lascoux \cite{Lascoux} and Butler \cite{butler}:

\begin{align}\label{eq:catabolizable HL formula}
        \mHL_\mu[X;q] &=  \sum\limits_{\substack{T\in \syt_n\\  \ctype(T)\unrhd\mu}}  q^{\cocharge(T)} s_{\shape(T)}(X).
\end{align}
The sum is over standard Young tableaux of size $n$ that satisfy a condition related to an operation called \textit{catabolism} (see Section \ref{section:cat} for definition). Catabolism, which was defined by Lascoux \cite{Lascoux},  is an operation on tableaux, or words, that lowers the cocharge value. 
Using catabolism, we can associate to each tableau $T$ a partition called its \textit{catabolizability type}, denoted $\ctype(T)$.
The sum in Equation \eqref{eq:catabolizable HL formula} is over all standard tableaux $T$ that have $\ctype(T) \unrhd \mu$, where $\unrhd$ denotes the dominance order on partitions.

In this paper, we give a monomial basis of $R_{\mu}$, where the monomials are indexed by standard tableaux that satisfy a catabolizability type condition. 
The set we construct is the natural subset of the descent basis to consider in light of the combinatorial formula \eqref{eq:catabolizable HL formula} for the graded Frobenius series.

For any permutation $w$, we associate a word $\cw(w)$ of the same length called the \textit{charge word} of $w$ such that the sum of the entries of $\cw(w)$ is charge($w$).
We refer to the monomial $\mathbf{x}^{\cw(w)}$, where the exponents are given by the charge word of $w$, as the \textit{charge monomial} corresponding to $w$.

\begin{thm1}
\makeatletter\def\@currentlabel{1.1}\makeatother
 Let $\mu$ be a partition of $n$.
     The set of charge monomials 
     \[ \{\mathbf{x}^{\cw(w)}\ | \ w\in S_n, \ \ctype(P(w)^t)\unrhd\mu\}\]
    is a monomial basis of $R_{\mu}$.
    In fact, it coincides with the basis given by Carlsson--Chou \cite{CC}.
\end{thm1}

Our basis consists of charge monomials corresponding to certain permutations in $S_n$. We later see that the set of charge monomials for all permutations in $S_n$ is equal to the Garsia-Stanton descent basis of $R_{1^n}$. Thus our set of monomials is indeed a subset of the descent basis.  

Although our basis is equal to that of Carlsson--Chou, our construction of it is different and provides the first direct connection between the structure of the Garsia-Procesi ring and the catabolizability formula for the modified Hall-Littlewood polynomial that gives its graded Frobenius character. 
It is immediate that our basis is of the correct size in each dimension, a fact that was nontrivial from the construction of Carlsson--Chou. 
The construction of our basis is clearly compatible with the Hilbert series of $R_{\mu}$ that we compute from \eqref{eq:catabolizable HL formula}:
\begin{align}
    \Hilbq(R_{\mu})= \sum\limits_{\substack{T\in \syt_n\\  \ctype(T)\unrhd\mu}}  q^{\cocharge(T)} |\syt(\shape(T))|
\end{align}
where $\syt(\lambda)$ denotes the set of standard Young tableaux of shape $\lambda$.

Furthermore, the relation between the charge monomial construction and the construction of Carlsson--Chou gives insight to how catabolizability type behaves under shuffling cocharge words  (see Sections \ref{subsection:charge}, \ref{section: coinvariant ring and garsia procesi} for definitions). In particular, we can show that there is a lower bound on the catabolizability type of the permutation we get from the shuffles.

\begin{thm2}\label{thm: sum of ctypes}
    Let $u^{(1)}, u^{(2)}, \dots, u^{(l)}, w$ be
    permutations such that $\cc(w)$ is a shuffle of $\cc(u^{(1)}),\dots, \cc(u^{(l)})$.
    We have
    \[\ctype(w) \unrhd \ctype(u^{(1)}) + \ctype( u^{(2)}) + \cdots+ \ctype(u^{(l)}),\]
    where $\ctype(u^{(1)}) + \ctype( u^{(2)}) + \cdots +  \ctype(u^{(l)})$ is the partition given by the partwise sum of $\ctype(u^{(1)}),\ctype( u^{(2)}),\dots,\ctype(u^{(l)})$.
\end{thm2}

Though it is easy to identify the ungraded Frobenius character of $R_\mu$ directly from its structure, as exhibited in Garsia--Procesi \cite{GP}, it is challenging to do the same for the graded character. This often requires heavy geometric machinery (\cite{deCon-Proc}) or further combinatorial properties of the ring and character (\cite{GP}).
However, our basis construction gives a way to easily identify the graded Frobenius character of $R_\mu$ as the modified Hall-Littlewood polynomial $\mHL_\mu[X;q]$ that only depends on \eqref{eq:catabolizable HL formula} and the ungraded Frobenius character.
For any Young subgroup $S_\gamma \subset S_n$, we define $N_{\gamma} = \sum_{\sigma\in S_\gamma} \sgn(\sigma) \sigma$ to be the antisymmetrizing element with respect to $S_\gamma$.
The graded Frobenius character $\Frobq(R_{\mu})$ is determined by $\Hilbq(N_\gamma R_{\mu}) = \langle e_\gamma, \Frobq(R_{\mu})\rangle$ for all Young subgroups $S_\gamma$, where $N_\gamma R_{\mu}$ is the antisymmetric part of $R_{\mu}$ with respect to $S_\gamma\subset S_n$.
We show that $\Hilbq(N_\gamma R_{\mu})$ is easily determined using our basis construction and the ungraded Frobenius character of $R_{\mu}$, by proving that the natural subset of our basis enumerated by $\langle e_\gamma, \Frobq(R_{\mu})\rangle$ has the property that applying $N_\gamma$ yields a basis for $N_\gamma R_{\mu}$.

\begin{prop2}\makeatletter\def\@currentlabel{B}\makeatother\label{prop2}
  Let $ \mu$ be a partition of $n$ and $\gamma = (\gamma_1,\dots, \gamma_l)$ be a composition of $n$. 
   The set
      \[ \{ N_\gamma \mathbf{x}^{\cw(w)} \ | \  w \in S_n, \ \ctype(P(w)^t) \unrhd\mu, \ \des(Q(w)) \subset \{\gamma_1,\gamma_1 + \gamma_2, \dots, \gamma_1 + \cdots + \gamma_{l-1}\}\}\]
    is a basis of $N_\gamma R_{\mu}$, where the pair $(P(w),Q(w))$ is the pair of standard tableaux in bijection with the permutation $w$ via the Schensted correspondence and $\des(Q(w))$ is the descent set of the tableau $Q(w)$.
\end{prop2}

The following corollary is immediate from Proposition \ref{prop2}  and the catabolizability formula \eqref{eq:catabolizable HL formula}.
\begin{cor3}
For any partition $\mu$ of $n$ we have $\Frobq(R_{\mu}) =\mHL_{\mu}[X;q].$
\end{cor3}

The paper is organized as follows. 
In Section \ref{section:background}, we review the definitions and  combinatorial tools we will use throughout the paper, including the Schensted correspondence, cocharge, and catabolism. 
In Section \ref{section: coinvariant ring and garsia procesi} we recall definitions regarding the coinvariant ring and the Garsia-Procesi rings and their monomial bases, as well as recall the construction of the Carlsson--Chou basis from \cite{CC}.
In Section \ref{section:descent basis}, we give the construction of our set and prove that it agrees with the Carlsson--Chou basis, which shows that our set is indeed a basis due to their results. In particular, we show that our set of monomials is contained in their basis and then conclude they are the same set since they have the same cardinality.
In Section \ref{section:ctype sum}, we introduce an algorithm that can be used to prove the other direction of containment, as well as other results about catabolizability types. 
In Section \ref{sec: remarks}, we use a result from the previous section to show that this basis gives us a simple description of the Hilbert series of $N_\gamma R_{\mu}$. From this, we show $\Frobq(R_{\mu})=\mHL_{\mu}[X;q]$.

\section{Background}\label{section:background}
\subsection{Words and Partitions} \label{subsec: words and parts}
For any word $w =  w_1w_2\dots w_n$ we write  $\rev(w) := w_n w_{n-1}\dots w_1$ for the reverse word.  A permutation $w\in S_n$ is a word of length $n$ where each element in $\{1,\dots, n\}$ appears exactly once.

The \textit{descent set} of $w\in S_n$ is defined to be $\des (w) := \{i \ | \ w_i > w_{i+1}\}.$
The statistic \textit{major index} (maj) is defined to be
$\text{maj}(w) = \sum_{i \in \des(w)} i.$
This statistic is Mahonian, meaning that it is distributed in the following way:
\begin{align*}
  \sum\limits_{w\in S_n} q^{\text{maj}(w)} = [n]_q!,
\end{align*}
where $[n]_q! = [1]_q[2]_q\cdots [n]_q$ with $[k]_q = 1 + q+\cdots + q^{k-1}$.

A \textit{partition} $\lambda$ of $n$, denoted $\lambda\vdash n$,  is given by $\lambda = (\lambda_1,\lambda_2,\dots,\lambda_l)$ where $\lambda_1\geq \lambda_2\geq \cdots \geq \lambda_l>0$ and $\lambda_1+\cdots + \lambda_l = n$. A \textit{composition} $\gamma$ of $n$, denoted $\gamma\models n$, is a sequence $(\gamma_1,\gamma_2,\dots,\gamma_l)$ where $\gamma_i >0$ and $\gamma_1 + \gamma_2 + \cdots+\gamma_l = n$. 

The \textit{Young diagram} of a partition $\lambda\vdash n$ is the partial grid of $n$ boxes, where we have $\lambda_i$ squares in the $i$th row. 
We use French notation, meaning that the first part corresponds to the bottom row.  For example, the partition $(3,1)$ corresponds to the Young diagram \ytableausetup{smalltableaux}\[\ydiagram{1,3}.\]
The \textit{transpose} of a partition $\lambda$, denoted $\lambda^t$, is the partition we get by switching the rows and columns of its Young diagram. 

There is a partial ordering on the set of all the partitions of size $n$ called the \textit{dominance} ordering, denoted $\unrhd$, defined by 
\[\mu\unrhd \lambda \ \  \Leftrightarrow \ \  \mu_1 + \cdots + \mu_k \geq \lambda_1+\cdots +  \lambda_k \text{ for all } k\]
It is well known that $\mu \unrhd \lambda \ \  \Leftrightarrow \ \   \mu^t \unlhd \lambda^t$. If we move a box of $\lambda$ to a lower row so that the resulting shape $\mu$ is a partition, we have $\mu \unrhd \lambda.$
Dominance is the transitive closure of moving boxes down. 

A \textit{semistandard (Young) tableau} of shape $\lambda$ is a filling of $\lambda$ with positive integers such that the rows are weakly increasing from left to right and the 
columns are strictly increasing from bottom to top. The \textit{weight} of a semistandard tableau $T$ is the tuple $(m_1, m_2,\dots)$  where $m_i$ is the number of times $i$ appears in $T$. We denote the shape of $T$ by $\shape(T)$.

A \textit{standard (Young) tableau} is a semistandard tableau where each entry in $\{1,2,\dots, n\}$ appears in the filling exactly once:  that is, it is a semistandard tableau of weight $(1^n)$. For any partition $\lambda, $ we denote the set of standard Young tableaux of shape $\lambda$ by $\syt(\lambda)$. We denote the set of standard Young tableaux of all shapes of size $n$ by $\syt_n$.

The \textit{reading word} $\rw(T)$ of a tableau $T$ is the word we get by concatenating the row words, where each row is read from left to right, from top to bottom. The transpose of a tableau $T$ with shape $\lambda$, denoted $T^t$, is the filling of $\lambda^t$ we get by swapping the rows and columns of $T$.

\subsection{Schensted correspondence}\label{subsec:RSK}
The Schensted correspondence gives a bijection from permutations $w\in S_n$ to pairs $(P(w),Q(w))$ of standard Young tableau of size $n$ of the same shape. We say $P(w)$ is the \textit{insertion tableau} and $Q(w)$ is the \textit{recording tableau} of $w$.
We note that for any given standard tableau $T$ of shape $\lambda$, the number of permutations $w$ with $P(w) = T$ is $|\syt(\shape(T))|$, since we have one such permutation for each possible choice of $Q(w)$.

We recall some properties of the Schensted correspondence that we will use in following sections. For more details, see \cite[Chapter 7]{EC2}. 
For a standard Young tableau $T$, we define the descent set of $T$ to be  $\des(T) = \{i  \ | \ i \text{ appears below } i+1 \text{ in } T\}.$
Note that $\des(T) = \des(\rw(T)^{-1})$.
\begin{prop}[{\cite[Lemma 7.23.11]{EC2}}]
    For any $w\in S_n$, we have $\des(w) = \des(Q(w))$. 
\end{prop}

\begin{prop}[{\cite[Corollary A1.2.11]{EC2}}]\label{thm:RSK rev}
For any $w\in S_n$, we have $P(\rev(w))= (P(w))^t.$
\end{prop}

\subsection{Symmetric functions}\label{subsec: symm functions}
For more detailed references on symmetric functions or the representation theory of the symmetric group, see \cite[Chapter 7]{EC2} or \cite{macdonald}.
Let $\Lambda$ denote the ring of symmetric functions in variables $X = (x_1,x_2,\dots, )$.
We follow Macdonald's notation \cite{macdonald} for the monomial symmetric functions $\{m_\lambda\}$, the elementary symmetric functions $\{e_\lambda\}$, the complete (homogeneous) symmetric functions $\{h_\lambda\}$, and the Schur functions $\{s_\lambda\}$.

The Schur functions $s_\lambda$ can be expanded in terms of the monomial symmetric functions:
\[s_\lambda  = \sum\limits_{\gamma \ \vdash |\lambda|} K_{\lambda,\gamma} m_{\gamma}\]
where for any partition $\gamma$, the coefficient $K_{\lambda,\gamma}$, which we call the \textit{Kostka number}, counts the number of semistandard Young tableaux of shape $\lambda$ and weight $\gamma$.
 These coefficients can also be expressed in terms of standard Young tableaux:

 \begin{prop}
     For two partitions $\lambda,\gamma = (\gamma_1,\gamma_2,\dots, \gamma_l)\vdash n$, we have
      \begin{align}\label{eq: kostka numb}
     K_{\lambda,\gamma} = | \{T\in \syt(\lambda)  \ | \ \des(T) \subset \{\gamma_1,\gamma_1 + \gamma_2, \dots, \gamma_1 +\cdots + \gamma_{l-1}\} \} |
 \end{align}
 \end{prop}
 \begin{proof}
     We know that $K_{\lambda,\gamma} = | \ssyt(\lambda,\gamma)|$, where $\ssyt(\lambda,\gamma)$ denotes the set of semistandard tableaux of shape $\lambda$, weight $\gamma$. For any $T\in \ssyt(\lambda,\gamma)$, we can construct a $S\in \syt(\lambda)$ by replacing all the $i$ in $T$, from left to right, with $(\gamma_1+\dots+\gamma_{i-1} +1),\dots,(\gamma_1+\dots+\gamma_{i})$. By construction, we have that the only possible descents in $S$ are in $\{\gamma_1,\gamma_1+\gamma_2,\dots, \gamma_1+\dots,\gamma_{l-1}\}$. We can see that this is a bijection by checking that we can reverse the construction and recover the original semistandard tableau.
 \end{proof}

A symmetric function $f$ is \textit{Schur-positive} if all the coefficients in the Schur expansion of $f$ are nonnegative integers. 

The \textit{Hall inner product} is the inner product defined on $\Lambda$ by the relation $\langle m_\lambda, h_\gamma\rangle = \delta_{\lambda,\gamma}$. We have $\langle f,g\rangle = \langle g,f\rangle$ for any $f,g\in \Lambda$.
The Schur functions are an orthonormal basis with respect to this inner product.
The involution $\omega$ on $\Lambda$ is defined by $\omega e_\lambda= h_\lambda$. This map $\omega$ is an isometry: that is, we have $\langle f,g\rangle = \langle \omega f,\omega g\rangle$ for any $f,g \in \Lambda$.
We also have $\omega s_\lambda = s_{\lambda^t}$ for any $\lambda$.

Recall that the irreducible representations of $S_n$ are indexed by partitions of size $n$.
The \textit{Frobenius characteristic map} is a map $F$ from (virtual) characters of $S_n$ to symmetric functions of degree $n$ that takes the irreducible character  indexed by $\lambda\vdash n$ to the Schur function $s_\lambda$. This map encodes the character of an $S_n$ representation $V$ into a Schur-positive symmetric function.
The \textit{Frobenius character} of $V$, denoted $\Frob(V)$, is defined to be $F(\chi_V)$, where $\chi_V$ is the character of $V$. We have $\Frob(V_\lambda) = F(\chi_{V_\lambda}) = s_\lambda$, where $V_\lambda$ is the irreducible $S_n$ representation indexed by $\lambda$.

 For a graded vector space $V = \oplus_{d\geq0} V_d$, the   \textit{Hilbert series} $\Hilbq(V)$ is 
\[\Hilbq(V) = \sum\limits_{d\geq 0} q^d \text{dim}(V_d).\]
For a graded $\mathbb{C}S_n$-module $V = \oplus_{d\geq0} V_d$, the \textit{graded Frobenius character} $\Frob_q(V)$ is 
\[\Frob_q(V) = \sum\limits_{d \geq 0} q^d \Frob(V_d).\]

For the Young subgroup $S_\gamma = S_{\gamma_1} \times \cdots \times S_{\gamma_l}\subset S_n$ corresponding to $\gamma\models n$, we define $N_{\gamma} = \sum_{\sigma\in S_\gamma} \sgn(\sigma) \sigma$ to be the antisymmetrizer with respect to $\gamma$. 
For any $\mathbb{C}S_n$-module $V$, the vector space $N_\gamma V$ is the subspace of elements of $V$ that are antisymmetric with respect to $S_\gamma$.
Let $\mathcal{E}\uparrow^{S_n}_{S_\gamma}$ denote the induction of the sign representation $\mathcal{E}$ of $S_\gamma$ to $S_n$. It is well known that $\Frob(\mathcal{E}\uparrow^{S_n}_{S_\gamma}) = e_\gamma$. Using Frobenius reciprocity (see \cite[Chapter 3.3]{Fulton-Harris}), we get the following result:
\begin{prop}
    For any $\mathbb{C}S_n$-module $V$ and Young subgroup $S_\gamma \subset S_n$, we have  $$\dim(N_\gamma V) = \langle e_\gamma, \Frob(V)\rangle.$$
\end{prop}
The analogous statement for graded modules holds as well:
\begin{prop}\label{prop: dim and antisym}
    For any graded $\mathbb{C}S_n$-module $V = \oplus_{d\geq0} V_d$ and Young subgroup $S_\gamma \subset S_n$, we have  $$\Hilbq(N_\gamma V) = \langle e_\gamma, \Frobq(V)\rangle.$$
\end{prop}

\subsection{Cocharge and Charge}\label{subsection:charge}
The \textit{transformed Hall-Littlewood polynomials} $H_\mu[X;q]$ are defined to be
\[H_\mu[X;q] = \sum K_{\lambda,\mu}(q) s_\lambda\]
where $K_{\lambda,\mu}(q)$ is the $q$-Kostka polynomial (see {\cite[Chapter III.6]{macdonald}} for definition). 
There is a combinatorial formula for $K_{\lambda,\mu}(q)$ due to Lascoux-Sch\"{u}tzenberger \cite{LS1} using a statistic defined on semistandard Young tableaux called \textit{charge}:
\[K_{\lambda,\mu}(q) = \sum\limits_{T\in \ssyt(\lambda,\mu)}  q^{\text{charge}(T)}.\] 
The sum is over all semistandard tableaux of shape $\lambda$ and weight $\mu$. In particular, $K_{\lambda,\mu}(1) = K_{\lambda,\mu}$.

The \textit{modified Hall-Littlewood polynomials} $\mHL_\mu [X;q]$ are defined by:
\[\mHL_\mu [X;q] = q^{n(\mu)} H_\mu[X;q^{-1}]\]
where $n(\mu) = \sum (i-1)\mu_i$.
Then we can write
\[ \tilde{H}_{\mu}[X;q] = \sum\limits_{\lambda}  \tilde{K}_{\lambda,\mu}(q) s_\lambda(X)\] when $\tilde{K}_{\lambda,\mu}(q): = q^{n(\mu)} K_{\lambda,\mu}(q^{-1})$ is the modified $q$-Kostka polynomial.
The polynomial $\tilde{K}_{\lambda,\mu}(q)$ can be expressed by the statistic \textit{cocharge} :\[\tilde{K}_{\lambda,\mu}(q) = \sum\limits_{T\in \ssyt(\lambda,\mu)}  q^{\text{cocharge}(T)}\]
where for a semistandard tableau of weight $\mu$ we have $\text{cocharge}(T) = n(\mu) - \text{charge}(T).$

We now define charge and cocharge on permutations.
For $w\in S_n$, we define $\text{charge}(w) = \text{maj}(\rev(w^{-1}))$ and $\text{cocharge}(w) = \binom{n}{2} - \text{charge}(w).$
Explicitly, we can compute $\text{cocharge}$ of $w$ by assigning labels to each letter in $w$.
The \textit{cocharge word} of $w$ (denoted $\cc(w)$) is a word of length $n$ that gives us the cocharge labelling of $w$, which we define in the following way. 
We label $1$ of $w$ with 0. We proceed by reading the numbers in increasing order: if we label $i$ with a $k$, then we label $(i+1)$ with a $k$ if it is to the right of $i$. We label $(i+1)$ with a $(k+1)$ if it is to the left of $i$.
The statistic $\textit{cocharge}(w)$  is the sum of the letters in $\cc(w)$. 

\begin{ex}\label{ex: cocharge word}
 Let $w = 3 \ \ 5 \ \ 1 \ \ 6\ \ 2 \ \ 4 \ \ 7 .$  
The corresponding cocharge word is $\cc(w)  = 1  \ \ 2 \ \ 0 \ \ 2 \ \ 0 \ \ 1 \ \ 2$, hence $\text{cocharge}(w) = 1 + 2 + 0 + 2 + 0 + 1 + 2 = 8.$
\end{ex}

We have the following classification of cocharge words of permutations.

\begin{prop}\label{prop: classification of cocharge words}
        Let $z$ be a word of nonnegative integers of length $n$ containing a 0.
        We have $z = \cc(w)$ for some $w\in S_n$ if and only if for all $i\in \{1,2,\dots,  n\}$, (at least) one of the following holds:
    \begin{enumerate}
        \item There exists $i'>i $ such that  $z_i =  z_{i'}$,
        \item There exists $i'<i $ such that  $z_i  + 1=  z_{i'}$ ,
         \item The entry $z_i$ is the maximal entry in $z$.
    \end{enumerate}

\end{prop}

\begin{proof}
    If $z = \cc(w)$, one of the three conditions will hold for any index $1\leq i \leq n$.
    The first (resp. second) case corresponds to when $w_i + 1$ appears to the right (resp. left) of $w_i$. 
   The last case corresponds to when $w_i$ has the highest cocharge value in the word.
   
   If we have such $z$ we can easily recover $w\in S_n$ such that $\cc(w) = z$. 
    For any nonnegative integer $k$, let $c_k$ denote the number of times $k$ appears in $z$ and $\{k_1<k_2<\cdots<k_{c_k}\} := \{i \ | \ z_i = k\}$. From condition (2), if $c_m >0$ but $c_{m+1} = 0$ for some nonnegative integer $m$, it follows that $z$ consists of $\{0,1,\dots, m\}$ and each of these letters appears at least once.

    We construct $w$ so that for each $k$ such that $c_k\neq 0$, the subword $w_{k_1}\dots w_{k_{c_k}}$ is the increasing sequence \[(c_0+\cdots+ c_{k-1} + 1)(c_0+\cdots+ c_{k-1} + 2)\cdots (c_0+\cdots+ c_{k-1} + c_k).\] 
    
    The resulting word $w$ is a permutation by construction. It is a simple exercise to check that $\cc(w) = z$.
    \end{proof}

\begin{lemma}\label{lemma: adding a repeated cocharge value is still a cc word}
Let $w\in S_n$, $x$ be a positive integer, and $1\leq k\leq n$. 
\begin{enumerate}[label=(\roman*)]
    \item Let $x = \cc(w)_i$ for some $i$. Then, any word $z$ we can get by inserting $x$ into $\cc(w)$ is a cocharge word of length $n+1$.
    \item Consider $m$ such that $\cc(w)_m$ is the largest, rightmost entry of $\cc(w)$. If $x = \cc(w)_m + 1$, any word we get by inserting $x$ to the left of position $m$ in $\cc(w)$ is a cocharge word of length $n+1$.
\end{enumerate}
\end{lemma}
\begin{proof}
Both cases follow from Proposition \ref{prop: classification of cocharge words}.
We know each entry in  $\cc(w)$ satisfies the condition in Proposition \ref{prop: classification of cocharge words}. We can also see that inserting a new letter into this word does not affect whether preexisting entries satisfy conditions (1) or (2), since it does not change the relative order of those entries within $z$. 
Hence it suffices to check that the new letter $x$ and the letters in $\cc(w)$ that only satisfied (3) still satisfy one of the three conditions when looking at $z$.
 Consider $m$ such that the letter $\cc(w)_m$ only satisfies (3). There is only one such letter: the largest, rightmost entry in $\cc(w)$. 

 For (i),  if $x = \cc(w)_i$ for some $i$, it follows that  $x \leq \cc(w)_m$, thus $\cc(w)_m$ is still the largest value in the word. Hence $\cc(w)_m$ satisfies (1) (if $x = \cc(w)_m$ and $x$ was inserted to the left of position $m$ in $\cc(w)$)  or (3) (otherwise). We can also see that $x$ clearly satisfies either (1) or (2).

 For (ii), we know $x = \cc(w)_m+1$. Hence $\cc(w)_m$ satisfies (2) and $x$ satisfies (3).   
\end{proof}

We define the \textit{charge} word $\cw(w)$ of permutation $w$ to be the following: 
\begin{align}\label{eq: charge is reverse of cocharge of reverse}
    \cw(w)  = \rev(\cc(\rev(w))).
\end{align}

We can see that $\cw(w)$ can be computed by the same algorithm we use to compute the cocharge word, except with left and right interchanged. The statistic $\textit{charge}(w)$  is the sum of the letters in $\cw(w)$. 

Lascoux-Sch\"{u}tzenberger \cite{LS1}  proved that charge is the unique statistic on words that satisfies a set of properties, one of them being the following:

\begin{thm}[Lascoux-Sch\"{u}tzenberger \cite{LS1}]
    If $w,w' \in S_n$ are such that $P(w) = P(w')$, then $\text{charge}(w) = \text{charge}(w')$.
\end{thm} 
Note that the original statement is  for words with partition weight that are Knuth equivalent; for our purposes, it suffices to consider permutations.
From this, we can define charge (resp. cocharge) on standard Young tableau $T$ by defining it on words and then extending the definition to tableaux by setting $\text{charge}(T) = \text{charge}(\text{rw}(T))$ (resp. $\text{cocharge}(T) = \text{cocharge}(\text{rw}(T))$).

One observation we can make from Proposition \ref{thm:RSK rev} and \eqref{eq: charge is reverse of cocharge of reverse} is the following:

\begin{prop}\label{thm: charge cocharge trabspose}
    If $T$ is a standard tableau, then
    $\text{charge}(T) = \text{cocharge}(T^t)$.
\end{prop}

To each $w \in S_n$, we associate a monomial $\mathbf{x}^{\cw(w)}$, where the exponents are given by the charge word of $w$. We refer to this monomial as the \textit{charge monomial }of $w$. 
\begin{ex}
Consider $w' = 7 \ \ 4 \ \ 2 \ \ 6\ \ 1 \ \ 5 \ \ 3 .$  Note that $w' = \rev(w)$ from Example \ref{ex: cocharge word}.
From \eqref{eq: charge is reverse of cocharge of reverse} , we know 
$\cw(w') = \rev(\cc(w)) = 2 \ \ 1 \ \ 0 \ \ 2 \ \ 0 \ \ 2 \ \ 1 $. Thus the corresponding charge monomial is $\mathbf{x}^{\cw(w')} = x_1^2 x_2 x_4^2x_6^2x_7.$
\end{ex}

\subsection{Catabolism and catabolizability}\label{section:cat}
Let $T$ be a standard tableau with $\text{rw}(T) = w w'$ where $w'$ is the first row of $T$. \textit{Catabolism} is defined to be the operation that takes $T = P(w w')$ to $K(T) := P(w' w)$. It is known that $\text{cocharge}(K(T)) = \text{cocharge}(T)- (|\lambda| - \lambda_1)$ where $\lambda = \shape(T)$. Repeated applications of catabolism eventually produce a one row tableau with zero cocharge.

We can see that applying catabolism to $T$ is equivalent to sliding the first row of $T$ to the right until it detaches from the higher rows, and then swapping the two pieces and applying jeu-de-taquin (for a definition of jeu-de-taquin, see \cite[Chapter 7.A1.2]{EC2}) to the resulting skew shape until we get a partition shape tableau. This is illustrated in the example below.

\begin{ex}  \ytableausetup{smalltableaux, aligntableaux= center} Consider the following tableau of shape $\lambda = (3,2,1)$: \[T = \ \ \ytableaushort{6,25,134}.\]
We can see that $\text{rw}(T) = 625134$ where $w = 625$ and $w' = 134$.  Observe that \[K(T) = P(w'w) =  P(134625) =  \ytableaushort{36,1245}.\]
Alternatively, we can use jeu-de-taquin to see
\begin{center}
$T = \ \ $\ytableaushort{6,25,134} $\rightarrow$ \ytableaushort{134,\none \none \none 6,\none \none \none 25}  $\xrightarrow[\text{jeu-de-taquin}]{}$ \ytableaushort{36,1245}  $\ \  = K(T).$
\end{center}
Note that $\cc(T) = 8$ and $\cc(K(T)) = 5 = 8 - (6-3)$, where $|\lambda| = 6, \lambda_1 = 3$.
\end{ex}

We use catabolism to define catabolizability type. For a standard tableau $T$, let $d(T)$ be the largest integer $m$ such that the first row of $T$ contains $1,2,\dots, m$.
The \textit{catabolizability type} of $T$ (denoted $\ctype(T)$) is given by the sequence
\begin{align}
    \ctype(T) = (d(T), d(K(T))-d(T), \dots, d(K^{i}(T)) - d(K^{i-1}(T)),\dots).
\end{align}

\begin{ex}\label{ex:cat using direct def}
    Consider \ytableausetup{smalltableaux}
$T =\  $\ytableaushort{68,45,1237}. Since $d(T) = 3$, we have $\ctype(T)_1 = 3$.
The result of applying catabolism is \[K(T) = P(12376845) = \ytableaushort{7,68,12345}.\]
Since $d(K(T)) = 5$, we have $\ctype(T)_2 = d(K(T)) - d(T) = 2$.
Repeating this process, we can compute $\ctype(T) = (3,2,1,1,1)$.
\end{ex}

We can think of $\ctype(T)$ as encoding the sequence of catabolisms we use to go from $T$ to the one row tableau. We have the following result regarding this sequence.
\begin{prop}[Lascoux {\cite[Lemma 9.6]{Lascoux}}]
    For standard tableau $T$, the sequence $\ctype(T)$ is a partition.
\end{prop}

An alternative way to view $\ctype(T)$ is using $m$-catabolizability (see  \cite{SW} for details).
Let $(m)$ denote the one row shape of length $m$. If $(m) \subset \shape (T)$ and the first row of $T$ contains the smallest $m$ letters in $T$, we define the \textit{$m$-catabolism} of $T$ to be deleting the first $m$ entries of the first row of $T$  and then proceeding as we did for standard catabolism, where we denote the resulting tableau $\text{Cat}_m(T)$. We say that $T$ of size $n$ is \textit{$\lambda$-catabolizable} for partition $\lambda = (\lambda_1, \lambda_2,\dots, \lambda_l)\vdash n$ if the first row of $T$ contains the smallest $\lambda_1$ many entries of $T$ and $\text{Cat}_{\lambda_1} (T)$ is $(\lambda_2,\dots, \lambda_l)$-catabolizable. 
Catabolizability type of a tableau and $\lambda$-catabolizability are related in the following way.

\begin{prop}[Shimozono-Weyman {\cite[Proposition 46]{SW}}]
    A standard tableau $T$ is $\lambda$-catabolizable if and only if $\ctype(T) \dom \lambda$.
\end{prop}

From this, we can compute $\ctype(T)$ by repeatedly applying maximal $m$-catabolisms,  where at each step we have \[\ctype(T)_i = d(\text{Cat}_{\ctype(T)_{i-1}}\cdots \text{Cat}_{\ctype(T)_{1}}(T)).\] This simplifies our process, since at each step we are dealing with smaller tableaux.

\begin{ex}
    Consider \ytableausetup{smalltableaux}
$T = $\ytableaushort{68,45,1237}.
We know $d(T) = 3$ and can compute $\text{Cat}_3(T)$:
\begin{align*}
     \ytableaushort{68,45,1237} \to  \ytableaushort{7,\none  68, \none 45} \xrightarrow[\text{jeu-de-taquin}]{}  \ytableaushort{7,68,45} = \text{Cat}_3(T).
\end{align*}
We see that $\ctype(T)_2 = d(\text{Cat}_3(T)) = 2$. 
Repeating this process, we can compute $\ctype(T) = (3,2,1,1,1)$, which is the same as what we get in Example \ref{ex:cat using direct def}. 
\end{ex}

In general, many things are unknown about catabolism and catabolizability, though the operation of catabolism itself is very easy to compute.
However, there is a concrete algorithm, due to Blasiak \cite{Blasiak}, called the \textit{catabolism insertion algorithm}, which computes the catabolizability type of a permutation $w$, where we define $\ctype(w) : = \ctype(P(w))$.

\begin{alg}[Blasiak {\cite[Algorithm 3.2]{Blasiak}}]\label{alg: cat insertion}
We define a function $f$ on pairs $(x,\nu)$, where $x = ya$ is a word ($a$ is the last letter of $x$) and $\nu$ is a partition, to be
\begin{align*}
    f(x,\nu) &= \begin{cases}
        (y,\nu + \epsilon_{a+1}) & \text{if }\nu + \epsilon_{a+1} \text{ is a partition,}\\ 
        ((a+1)y,\nu) & \text{otherwise.}
    \end{cases}
\end{align*}
where $\epsilon_{a+1} $ is the composition $(0,\dots,0, 1,0,\dots)$ with 1 in the $(a+1)$th coordinate and $\nu + \epsilon_{a+1})$ is the component wise sum of the two compositions.
 
Let $w$ be a permutation of length $n$ with cocharge word $\cc(w)$. We apply $f$ to $(\cc(w),\emptyset)$ repeatedly until we get $(\emptyset, \mu)$, where $\mu \vdash n$.
\end{alg}
We introduce terminology regarding this algorithm.
When we refer to the $i$th letter of the input word, we also count the letters that we have deleted from the word. Hence the input word is always of length $n$, but potentially with empty spots.
At each step of the algorithm, when the input is $(ya,\nu)$, we say we \textit{read} $i$ if the letter $a$ corresponds to the $i$th letter in the original input word $\cc(w)_i$.

\begin{prop}[Blasiak {\cite{Blasiak}}]
   When applying Algorithm \ref{alg: cat insertion} to $(\cc(w),\emptyset),$ the resulting partition $\mu $ is $\ctype(w)$.
\end{prop}

As a slight modification to the original algorithm introduced in \cite{Blasiak}, we record the entries of $\cc(w)$ that correspond to the boxes in $\mu$ as we build the partition. When reading $i$ results in adding a box to the partition $\nu$, we fill the new box with $i$. This gives us a standard filling (but not a SYT) of shape $\ctype(w)$, meaning $\{1,\dots, n\}$ appear exactly once. We denote this filling as $T_w$. For the modified algorithm, the final tuple is $(\emptyset, T_w)$.

If the initial word $z$ is not a cocharge word, this algorithm may not terminate. For example: if $z = 02$, the second letter will never be deleted, since we cannot add a box to the partition $(1)$ in rows $2$ or higher.

\begin{ex}\label{ex: cat alg}
    Consider \begin{align*}
        w &=  6 \ 3 \ 4  \ 1 \ 2\ 5  \\ 
        \cc(w) &=  2 \ 1 \ 1 \ 0 \ 0 \ 1 .
    \end{align*}
    We will apply $f$ to $(\cc(w),\emptyset)$ repeatedly until we get an empty word in the first coordinate.
In order to keep track of the indices, we do not rotate $\cc(w)$: instead, we read the word from right to left. The position we are reading at each step of the algorithm is underlined. 
     \ytableausetup{smalltableaux, centertableaux}
     {\allowdisplaybreaks\begin{align*}
          ( 2 \ 1 \ 1 \ 0 \ 0 \ \mathbf{\underline{1}} \hspace{.1cm}&, \hspace{.3cm} \emptyset) \\ &\downarrow\\ 
     ( 2 \ 1 \ 1 \ 0 \ \mathbf{\underline{0}} \ 2\hspace{.1cm} &, \hspace{.3cm}\emptyset \hspace{.1cm}) \\ &\downarrow\\ 
     ( 2 \ 1 \ 1 \  \mathbf{\underline{0}} \ \hspace{.2cm} \ 2 \hspace{.1cm}&, \hspace{.3cm} \ytableaushort{5} \hspace{.1cm}) 
      \\ &\downarrow\\ 
     ( 2 \ 1 \ \mathbf{\underline{1}} \  \hspace{.2cm} \ \hspace{.2cm} \ 2 \hspace{.1cm}&,\hspace{.1cm} \ytableaushort{54}\hspace{.1cm}) \\ &\downarrow\\ 
     ( 2 \ \mathbf{\underline{1}} \  \hspace{.2cm}  \  \hspace{.2cm} \ \hspace{.2cm} \ 2\hspace{.1cm} &, \hspace{.3cm}\ytableaushort{3,54}\hspace{.1cm}) 
      \\ &\downarrow\\ 
     ( \mathbf{\underline{2}} \ \hspace{.2cm} \  \hspace{.2cm}  \  \hspace{.2cm} \ \hspace{.2cm} \ 2 \hspace{.1cm}&, \hspace{.3cm}\ytableaushort{32,54}\hspace{.1cm}) 
     \\ &\downarrow\\ 
     ( \hspace{.2cm}\ \hspace{.2cm} \  \hspace{.2cm}  \  \hspace{.2cm} \ \hspace{.2cm} \ \mathbf{\underline{2}}  \hspace{.1cm}&,\hspace{.3cm} \ytableaushort{1,32,54}\hspace{.1cm}) 
 \\ &\downarrow\\ 
     ( \hspace{.2cm}\ \hspace{.2cm} \  \hspace{.2cm}  \  \hspace{.2cm} \ \hspace{.2cm} \ \hspace{.2cm} \hspace{.1cm} &, \hspace{.3cm}\ytableaushort{16,32,54}\hspace{.1cm}) 
    \end{align*}}
    From this, we conclude $T_{634125} = \ytableaushort{16,32,54}$ and $\ctype(634125) = (2,2,2)$.
\end{ex}
We now state some nice properties of the filling $T_w$. 

\begin{lemma}\label{lemma:row is sum of cocharge label and mult}
    If we add $i$ to row $r$ of $T_w$ when reading $i$ for the $k$th time, we have
    \[k + cc(w)_i = r.\]
\end{lemma}
\begin{proof}
    Each time we read $i$ without adding it to $T_w$, we increase the $i$th letter of the input word by $1$. 
    After reading through the word $(k-1)$ times without adding $i$ to the filling, the $i$th letter in the input word is $\cc(w)_i + (k-1)$.
    Since we add $i$ the next time we read the $i$th letter, we know we add it to the $(\cc(w)_i + (k-1)+1)$th row, where the last $1$ is coming from the fact that $f(x,\nu) = (y,\nu +\epsilon_{a+1})$. 
\end{proof}

We now look at subwords of $\cc(w)$ determined by the columns of $T_w$. Consider $w\in S_n$ such that $\ctype(w) = \mu$.
For $1\leq j \leq \mu_1$ and $1\leq r \leq \mu'_j$, let \[\{i \ | \ i \text{ is in row } r', \text{ column } j \text{ of } T_w, \ 1\leq r' \leq r\} = \{j^{(r)}_1 < j^{(r)}_2 < \cdots < j^{(r)}_r\}.\]

We are numbering the entries in column $j$ of $T_w$ in increasing order, though that order may not match the order in which they appear in the column. 

We define the subword $\cc(w)^{(j,r)}$ of $\cc(w)$ to be \[\cc(w)^{(j,r)} = \cc(w)_{j^{(r)}_1}\cc(w)_{j^{(r)}_2}\dots \cc(w)_{j^{(r)}_r}.\]
In this case, $\cc(w)^{(j,r)}$ is the subword corresponding to the first $r$ rows of column $j$ in $T_w$. We write $\cc(w)^{(j)} = \cc(w)^{(j,\mu_j)}$ when $r$ is equal to the height of column $j$.

\begin{ex}
   Using $w=  634125 $ from Example \ref{ex: cat alg}, we can see that \[\cc(w)^{(1,3)} = \cc(w)_1\cc(w)_3\cc(w)_5 = 210.\]
We can also see that $j^{(3)}_1 = 1$ but $1$ is not in the first row of $T_w$. 
\end{ex}

\begin{prop}\label{prop: subword of column is a cocharge word}
For $w\in S_n$ such that $\ctype(w) = \mu$, consider $j,r$ such that $1\leq j \leq \mu_1$ and $1\leq r \leq \mu'_j$.
Then $\cc(w)^{(j,r)} = \cc(\sigma)$ for some $\sigma \in S_r$.
\end{prop}

\begin{proof}
    We show this by induction on $r$. When $r=1$, we know $\cc(w)^{(j,1)} = 0$, since for any $i$ that appears in the first row of $T_w$ we must have $\cc(w)_i = 0$. Hence the claim holds.

    Assume the claim holds for $(r-1)$. We know that we obtain $\cc(w)^{(j,r)}$ from $\cc(w)^{(j,r-1)}$ by adding $\cc(w)_a$ into $\cc(w)^{(j,r-1)}$ in the correct position, where $a$ is the entry in row $r$, column $j$ of $T_w$.
    Let $b$ denote the entry directly below $a$ in $T_w$.
    Since $b$ appears below $a$, we know that when we constructed $T_w$ using Algorithm \ref{alg: cat insertion} to $\cc(w)$ we first added $b$, and then $a$, to the filling.
    If we let $k_a, k_b$ denote the number of times we read $a,b$ before adding them to $T_w$, we must have $k_b \leq  k_a$.
    If $k_a = k_b$, then since we add $b$ before $a$, we must have $a<b$, since $\cc(w)_b$ must be to the right of $\cc(w)_a$ in $\cc(w)$. We also have $\cc(w)_{a} = \cc(w)_{b} + 1$ from Lemma \ref{lemma:row is sum of cocharge label and mult} since $a$ is one row above $b$.  If $k_b < k_a$, then $\cc(w)_{b} \geq  \cc(w)_{a}$. 
    
    From this, we know that if $\cc(w)_a$ is larger than every letter in $\cc(w)^{(j,r-1)}$, we have $\cc(w)_a = \cc(w)_b+1$ where $b>a$ is the entry directly below $a$ in $T_w$ and $\cc(w)_b$ is the largest letter in $\cc(w)^{(j,r-1)}$.

     By Lemma \ref{lemma: adding a repeated cocharge value is still a cc word}, we conclude $\cc(w)^{(j,r)}$ is a cocharge word by induction.
\end{proof}

\begin{rem}
    The fact that $\cc(w)^{(j,r)}$ is a cocharge word is not obvious a priori, since it is not true that any subword of a cocharge word is a cocharge word. For example, consider $w = 21 $, which has $\cc(w) =10.$ Though $1$ is a subword of $\cc(w)$, it is not a cocharge word.
\end{rem}

\section{Coinvariant Ring and Garsia-Procesi Ring}\label{section: coinvariant ring and garsia procesi}

The \textit{coinvariant ring} $R_{1^n}$ is defined to be  $R_{1^n}= \mathbb{C}[\mathbf{x}]/ I,$ where $I$ is the ideal generated by $S_n$ invariant polynomials with zero constant term. More specifically, we can write $I = \langle e_1(\mathbf{x}),\dots, e_n(\mathbf{x})\rangle$.
It is clear that $R_{1^n}$ is a $\mathbb{C}S_n$-module under the action permuting the variables. 
It is well-known that this is a graded version of the regular representation of $S_n$: in particular, $\dim(R_{1^n}) = n!$. As a graded algebra with $S_n$ action, it is isomorphic to $H^{\ast}(\mathcal{F}_n)$, where $\mathcal{F}_n$ is the complete flag variety for $GL_n$. In particular,we have $\Hilbq(R_{1^n}) =[n]_q!$.

In Type $A$, we have that $H^{\ast}(\mathcal{F}_\mu)$,  the cohomology ring of the Springer fiber indexed by $\mu\vdash n$, is a quotient of $H^{\ast}(\mathcal{F}_n)$. These cohomology rings also have an $S_n$ action: we have the following concrete presentation of these $\mathbb{C}S_n$-modules as quotients of $R_{1^n}$ due to De Concini-Procesi \cite{deCon-Proc} and Tanisaki \cite{Tanisaki}. 
 Consider the transpose partition $\mu^{t} = (\mu^t_1 \geq \mu^t_2 \geq \cdots \geq  \mu^t_n \geq  0)$, where we pad the end of the partition with $0$'s until we have a tuple of length $n$. Let $p_k (\mu) : =\mu^t_n + \cdots + \mu^t_{n-k+1}$. Note that $p_k(\mu)$ is the number of boxes of the Young diagram of $\mu$ that are not in the first $n-k$ columns.  Given a subset $S\subset \mathbf{x}$, we define $e_d(S)$ to be the $d$th elementary symmetric function in the set of variables  $S$. 
The \textit{Tanisaki ideal} $I_\mu$ is defined to be
\[I_\mu = \langle e_d(S) \ | \ S\subset \mathbf{x}, d> |S| -p_{n-|S|}(\mu) \rangle .\]
The \textit{Garsia-Procesi ring} $R_\mu$ is the quotient  $R_\mu = \mathbb{C}[\mathbf{x}]/I_\mu$.
As a graded $\mathbb{C}S_n$-module, it is isomorphic to $H^{\ast}(\mathcal{F}_\mu)$.
The graded Frobenius character of $R_\mu$ is the modified Hall-Littlewood polynomial $\mHL_{\mu}[\mathbf{x};q]$.
Using the combinatorial formula \eqref{eq:catabolizable HL formula} for $\mHL_{\mu}[\mathbf{x};q]$, we compute $\Hilbq(R_\mu)$:
\begin{align}
     \Hilbq(R_{\mu}) &= \sum\limits_{\substack{T\in \syt\\  \ctype(T)\unrhd\mu}} q^{\text{cocharge}(T)} \text{dim} (V^{(\shape(T))}) \nonumber\\&=  \sum\limits_{\substack{T\in \syt\\  \ctype(T)\unrhd\mu}}  q^{\text{cocharge}(T)} |\syt(\shape(T))| \nonumber \\&=  \sum\limits_{\substack{T\in \syt\\  \ctype(P(w))\unrhd\mu}}  q^{\text{cocharge}(w)}\label{eq:Hilb of R_mu}.
\end{align}
When $\mu = 1^n$, we have $I_{1^n} = \langle e_1(\mathbf{x}),\dots, e_n(\mathbf{x})\rangle,$ hence the Garsia-Procesi ring indexed by $1^n$ is the full coinvariant ring $R_{1^n}$.

\subsection{Monomial bases of {$R_{1^n}$} and generalizations to {$R_{\mu}$}}
We recall two monomial bases of the coinvariant ring, both of which are indexed by permutations in $S_n$.

The \textit{Artin basis} is given by the following set of monomials:

\[\{f_\sigma(\mathbf{x}) = \prod_{i<j, \sigma_i > \sigma_j} x_{\sigma_i} \  | \ \sigma \in S_n \} = \{x_1^{a_1}x_2^{a_2}\cdots x_n^{a_n} \ | \ a_i < i\}.\]
From the second description of the Artin monomials, we can easily see that \[\sum_{w\in S_n}q^{\text{deg}(f_w(\mathbf{x}))} = (1+q)(1+q+q^2) \cdots (1+q+q^2+\cdots + q^{n-1}) = [n]_q!.\]

The \textit{Garsia-Stanton descent basis} is given by the following set of monomials:
\[\{g_w (\mathbf{x}) = \prod_{i:w_i > w_{i+1}} x_{w_1} \cdots x_{w_i} \ | \ w \in S_n\}.\]
Let $D_n = \{\alpha : \mathbf{x}^{\alpha} = g_w(\mathbf{x}) \text{ for some } w \in S_n\}$ denote the set of exponents of monomials in the descent basis. We refer to these words as \textit{descent words.} 
\begin{lemma}\label{lemma:charge and majt}
    We have $D_n = \{c(w) \ | \ w \in S_n\}.$
\end{lemma}
\begin{proof}
    For any $\sigma\in S_n$, the charge monomial is  \[\mathbf{x}^{\cw(\sigma)} = \prod_{j:(\sigma^{-1})_{j} < (\sigma^{-1})_{j+1}} x_{(\sigma^{-1})_{j+1}} \cdots x_{(\sigma^{-1})_n}.\]
Thus $\mathbf{x}^{\cw(\sigma)} = g_w(\mathbf{x})$ where $w = \text{rev}(\sigma^{-1})$. 
\end{proof}

The \textit{descent order} on monomials is defined by:
\begin{align}\label{eq:descent order}
    \mathbf{x}^{\alpha} \leq_{\text{des}}\mathbf{x}^{\beta} \Leftrightarrow \begin{cases}
    \text{sort}(\alpha) \preceq \text{sort}(\beta),  \\ \text{sort}(\alpha) = \text{sort}(\beta) \text{ and } \alpha \preceq \beta,
\end{cases}
\end{align}
where sort$(\alpha)$ reorders the vector $\alpha$ into a partition and $\preceq$ is the lexicographic ordering on words.

The descent monomials correspond to the permutation statistic $\text{maj}$. That is, we have $\text{deg}(g_w(\mathbf{x})) = \text{maj}(w)$. 
Since this statistic is Mahonian, we have:
 \[\sum\limits_{w\in S_n} q^{\text{deg}(g_{w}(\mathbf{x}))} = \sum\limits_{w\in S_n} q^{\text{maj}(w)} = [n]_q!.\]
Thus, either basis gives a combinatorial explanation of $\Hilbq(R_{1^n}) = [n]_q!$. 

Since $R_\mu$  is a quotient of $R_{1^n}$, the natural question to ask if whether there exists a subset of the Artin monomials or the descent monomials which is a monomial basis of $R_\mu$.
Garsia-Procesi \cite{GP} constructed a monomial basis of $R_\mu$  which is a subset of the Artin basis of $R_{1^n}$. However, this basis does not give an analogous combinatorial explanation of the expression \eqref{eq:Hilb of R_mu}.

The problem of finding a subset of the descent basis that gives a monomial basis of $R_\mu$ was recently solved by Carlsson--Chou \cite{CC}.
We now review the construction of the Carlsson--Chou descent basis. 
For two words $z^{(1)},z^{(2)}$ of length $l_1, l_2$ we define $\text{Sh}(z^{(1)},z^{(2)})$ to be the collection of all words $u$ of length $(l_1+l_2)$ such that $z^{(1)}$ and $z^{(2)}$ are two disjoint subwords of $u$.
We construct such $u$ by interleaving $z^{(1)}$ and  $z^{(2)}$. An element $u$ of $\text{Sh}( z^{(1)},z^{(2)})$ is a \textit{shuffle} of $z^{(1)}$ and $z^{(2)}$. We can extend this definition and let $\text{Sh}(z^{(1)},\dots , z^{(l)})$ denote the set of all shuffles of $z^{(1)},\dots , z^{(l)}$.
We also define the set of \textit{reverse shuffles}  of  $z^{(1)},\dots , z^{(l)}$, denoted $\text{Sh}'(z^{(1)},\dots , z^{(l)})$, to be equal to $\text{Sh}(\rev (z^{(1)}),\dots , \rev(z^{(l)}))$.

The Carlsson--Chou descent basis of the Garsia-Procesi ring is defined as follows. 
For $\mu = (\mu_1,\dots,\mu_l) \vdash n$, define $\mathcal{D}_\mu$ to be
\begin{align}
    \mathcal{D}_\mu = \bigcup_{(z^{(1)},\dots , z^{(l)})} \sh (z^{(1)},\dots , z^{(l)})
\end{align}
where $(z^{(1)},\dots , z^{(l)})$ ranges over all $l$-tuples in $\mathcal{D}_{\mu_1}\times \cdots \times \mathcal{D}_{\mu_l}$.

\begin{thm}[Carlsson--Chou \cite{CC}]
The set $\mathbf{x}^{\mathcal{D}_\mu} = \{\mathbf{x}^\alpha \ | \ \alpha \in  \mathcal{D}_\mu\}$  is a monomial basis of $R_{\mu^t}$. 
\end{thm}   

\begin{ex}\label{ex: 31}
    Consider $\mu = (3,1).$ Note that 
    \begin{align*}
        D_3 &= \{012,011, 101, 001,010, 000\},\\
        D_1 &= \{0\}.
    \end{align*}
  Then we have 
  \[\mathcal{D}_{3,1}  = \{0012,\ 0102,\ 0120, \ 0011, \ 0101, \ 1001, \ 1010, \ 0110,\  0001,\ 0010,\ 0100, \ 0000\} .\]
  The corresponding set of monomials is  
  \begin{align*}
       \{x_3x_4^2,
  \ x_2x_4^2, \ x_2x_3^2, \ x_3x_4 ,\  x_2x_4, \ x_1x_4, \ x_1x_3, \ x_2x_3, \ x_4, \ x_3 , \ x_2, \ 1\},
  \end{align*}
\end{ex}

\begin{rem}
    Carlsson--Chou show that their basis is a subset of the Garsia-Stanton descent basis. 
However, it is not obvious how to see directly for which $w\in S_n$ we have $g_w(\mathbf{x}) \in \mathbf{x}^{\mathcal{D}_{\mu}}$ without computing $\mathcal{D}_\mu$.
It is also not obvious that $|\mathcal{D}_{\mu}| = \dim(R_{\mu^t})$, since multiple shuffles can correspond to the same descent word. For example, $\mathbf{0}101 \in \sh(101,\mathbf{0})$ and $01\mathbf{0}1 \in \sh(011,\mathbf{0})$ both correspond to the same element in $\mathcal{D}_{\mu}$.

\end{rem}

\section{Charge Monomial Basis of the Garsia-Procesi ring}\label{section:descent basis}
We now define sets of charge monomials that are monomial bases of Garsia-Procesi rings.
Our construction involves charge words of permutations whose insertion tableaux satisfy a catabolizability condition.

\begin{defn}
    The set $\mathcal{C}_\mu$ is defined to be
\begin{align}\label{def:C_mu}
    \mathcal{C}_\mu := \{\cw(w)\ | \ w\in S_n, \  \ctype(P(w) ^t)\unrhd\mu\}.
\end{align}
\end{defn}

From Lemma \ref{lemma:charge and majt}, we know $\mathcal{C}_\mu\subset D_n$.

\begin{thm1}
\makeatletter\def\@currentlabel{A}\makeatother
\label{thm: desc basis are the same} 
    The set 
    \begin{align}\label{eq: our basis}
        \mathbf{x}^{\mathcal{C}_\mu} = \{\mathbf{x}^\alpha \ | \ \alpha \in \mathcal{C}_\mu\}
    \end{align}
     is a monomial basis of $R_{\mu}$. 
       In fact, it coincides with the basis $  \mathbf{x}^{\mathcal{D}_{\mu^t}}$ given by Carlsson--Chou \cite{CC}, ie., $\mathcal{C}_\mu = \mathcal{D}_{\mu^t}$.
\end{thm1}

Before we prove Theorem \ref{thm: desc basis are the same}, we point out the connections between this basis and the Hilbert series $\Hilbq(R_{\mu})$.
Using combinatorial formula \eqref{eq:catabolizable HL formula} for the modified Hall-Littlewood polynomials and Proposition \ref{thm: charge cocharge trabspose} gives us the following expression for $\mHL_{\mu}[X;q]$.
\begin{align}
    \mHL_{\mu} [X;q] &=  \sum\limits_{\substack{T\in \syt_n\\  \ctype(T^t)\unrhd\mu}}  q^{\charge(T)} s_{\shape(T^t)}(X).
\end{align}

From this, we can see that the classification of insertion tableaux that appears in \eqref{def:C_mu} is the natural one to consider when looking at subsets of descent monomials. 
In fact, it is evident from our construction that the degrees of the monomials in our set match what we expect from $\Hilbq(R_{\mu})$. 
Rephrasing \eqref{eq:Hilb of R_mu} for $\mHL_{\mu}[X;q]$ using Proposition \ref{thm: charge cocharge trabspose}, we have
\begin{align}
     \Hilbq(R_{\mu}) &=  \sum\limits_{\substack{w\in S_n\\  \ctype(P(w)^t)\unrhd\mu}}  q^{\text{charge}(w)} .
\end{align}
It is also apparent that $|\mathcal{C}_\mu| = \dim(R_{\mu})$, which was not obvious from the definition of $\mathcal{D}_{\mu^t}$.
Furthermore, it is clear that our set has the correct number of monomials of each degree to be a basis of $R_{\mu}$ by construction.
That is, for any nonnegative $d$, we have 
\[|\{w\in S_n, \ctype(P(w))^t \unrhd \mu, \charge(w) = d\}| = \dim((R_{\mu})_d),\]
where $(R_{\mu})_d$ is the degree $d$ component of $R_{\mu}$.
This construction also gives us a direct way to determine for which $w\in S_n$ we have $g_w(\mathbf{x})\in \mathbf{x}^{\mathcal{C}_\mu}$.
In particular, we have $g_w(\mathbf{x})\in \mathbf{x}^{\mathcal{C}_\mu}$ for $w = \text{rev}(\sigma^{-1})$ where $\ctype(P(\sigma)^t)\unrhd \mu$.

\begin{ex}
Consider $\mu = (2,1,1)$. There are five standard tableaux $S$ such that $\ctype(S^t) \unrhd (2,1,1) = \mu$. We list them, along with all words $w$ such that $P(w) = S$ and their charge monomials:
\begin{center}
     \ytableausetup{smalltableaux,  aligntableaux = center}
   \begin{tabular}{ p{1cm}|p{5cm}  p{5cm}  }
$S$ & $\{w \ | \ P(w) = S\}$ & $\{x^{\cw(w)} \ | \ P(w) = S\}$\\
 \hline
\\[-.5em]
 \ytableaushort{2,134}  &\{2134,
  \ 2314, \ 2341\}& $\{x_3x_4^2,
  \ x_2x_4^2, \ x_2x_3^2\}$ \\[-.5em] \\
\ytableaushort{24,13} &  \{2143 ,\  2413\} &  $\{x_3x_4 ,\  x_2x_4\}$  \\[-.5em] \\
\ytableaushort{4,2,13} & \{4213, \ 4231, \ 2431\} &  $\{x_1x_4, \ x_1x_3, \ x_2x_3\}$\\[-.5em] \\
\ytableaushort{3,2,14} & \{3214, \ 3241 , \ 3421\} & $\{x_4, \ x_3 , \ x_2\}$\\[-.5em] \\
\ytableaushort{4,3,2,1}&  \{4321\}& $\{1\}$
\end{tabular}
\end{center}
Note that all charge monomials for words with the same insertion tableau $S$ have degree equal to charge$(S)$. 
The resulting set of charge monomials is  \[\{x_3x_4^2,
  \ x_2x_4^2, \ x_2x_3^2, \ x_3x_4 ,\  x_2x_4, \ x_1x_4, \ x_1x_3, \ x_2x_3, \ x_4, \ x_3 , \ x_2, \ 1\},\]
which is the same as in Example \ref{ex: 31}.
\end{ex}

Now we will prove Theorem \ref{thm: desc basis are the same} by showing $\mathcal{C_\mu} = \mathcal{D}_{\mu^t}$, which implies that our set is a basis of $R_{\mu}$ by the work of Carlsson--Chou.
Since we know that $\mathcal{D}_{\mu^t}$ is a basis of $R_{\mu}$ and that $\mathcal{C_\mu}$ has the correct cardinality to be a basis, it suffices to show that $\mathcal{C}_\mu\subset \mathcal{D}_{\mu^t}$.
We first make the following observation, which is a corollary of Lemma \ref{lemma:charge and majt} using \eqref{eq: charge is reverse of cocharge of reverse}:

\begin{cor}\label{cor:cocharge and majt} $D_n = \{\rev(\cc(w)) \ | \ w \in S_n\}.$
\end{cor}

To show that $\mathcal{C}_\mu\subset \mathcal{D}_{\mu^t}$, we translate the problem into one about cocharge words to make use of Algorithm \ref{alg: cat insertion}. We first consider the case where $\ctype(w) = \mu$.

\begin{prop}\label{prop: ctype implies shuffle}
    Let $\mu$ be a partition of $n$ of length $l$ and  $w$ be a permutation of $n$ with $\ctype(w) = \mu$. 
    Then $\cc(w) \in \sh'(z^{(1)},\dots , z^{(l)})$ for some $(z^{(1)},\dots , z^{(l)})\in D_{\mu^t_1}\times \cdots \times D_{\mu^t_l}$.
\end{prop}

\begin{proof}
    We obtain a filling $T_w$ of shape $\ctype(w) = \mu $ by applying Algorithm \ref{alg: cat insertion} to $\cc(w)$.
    It is clear that $\cc(w) \in \sh (\cc(w)^{(1)}, \cc(w)^{(2)}, \dots, \cc(w)^{(l)})$, where $\cc(w)^{(j)}$ is the subword of $\cc(w)$ corresponding to column $j$ in $T_w$. We know $\cc(W)^{(j)}$ has length $\mu_j$, which is the height of column $j$. 
    Thus, from Proposition \ref{prop: subword of column is a cocharge word}, we know $\cc(w)^{(j)} = \cc(\sigma)$ for $\sigma\in S_{\mu^t_j}$. 
    By Corollary \ref{cor:cocharge and majt}, this implies  $\rev(\cc(w)^{(j)}) \in \mathcal{D}_{\mu^t_j}$.
\end{proof}

\begin{ex}\label{ex: cat ex}
    Recall that in Example \ref{ex: cat alg}, we showed that for $w = 6 \ 3 \ 4  \ 1 \ 2\ 5 $, with $\cc(w) = 2 \ 1 \ 1  \ 0 \ 0\ 1 $, we have  $\ctype(w) = (2,2,2).$ and that the filling $T_w$ of $ (2,2,2) $ is 
        \[\ytableaushort{16,32,54}. \] 
    From this, we construct two subwords of lengths $3$ respectively:
    \begin{align*}
        \cc(w)^{(1)} &= \cc(w)_1\cc(w)_3\cc(w)_5 = 210\\
         \cc(w)^{(2)} &= \cc(w)_2\cc(w)_4\cc(w)_6 = 101
    \end{align*}
    We can see that $\cc(w)^{(1)} = \cc(321) = \rev(\cw(123))$ and $\cc(w)^{(2)} = \cc(213) = \rev(\cw(213))$. 
    Thus $\cc(w)\in \text{Sh} '(\cw(123), \cw(213))$ , where $(\cw(123), \cw(213)) \in \mathcal{D}_{(3,3)}$.
\end{ex}

We rephrase Proposition \ref{prop: ctype implies shuffle} to match the language of Theorem \ref{thm: desc basis are the same} using \eqref{eq: charge is reverse of cocharge of reverse}:

\begin{cor}\label{cor:ctype implies shuffle}
    For any $w\in S_n$ such that $\ctype(P(w)^t) =\mu$, we have $\cw(w)\in \mathcal{D}_{\mu^t}$.
\end{cor}
\begin{proof}
    From  \eqref{eq: charge is reverse of cocharge of reverse} we know  \[\cc(w) \in \sh'(z^{(1)},\dots , z^{(l)}) \Leftrightarrow\cw(\rev(w))= \rev(\cc(w)) \in \sh(z^{(1)},\dots , z^{(l)}).\]
    We also have $\ctype(\rev(w)) = \ctype(P(w)^t)$ using Theorem \ref{thm:RSK rev}.
\end{proof}

Note that Corollary \ref{cor:ctype implies shuffle} only considers permutations $w$ with $\ctype(P(w))^t) = \mu$. To prove the inclusion $\mathcal{C}_\mu \subset \mathcal{D}_{\mu^t}$, we extend this argument to consider $\ctype(P(w))^t) \rhd  \mu$.

\begin{proof}[Proof of  Theorem \ref{thm: desc basis are the same}]
We know $|\mathcal{C}_\mu| = |\mathcal{D}_{\mu^t}|$, since $|\mathcal{C}_\mu| = \text{dim}(R_{\mu})$ and $\mathcal{D}_{\mu^t}$ is a basis of $R_{\mu}$.
Hence to show equality of the two sets, we will show $\mathcal{C}_\mu \subset \mathcal{D}_{\mu^t}$.
From Corollary \ref{cor:ctype implies shuffle}, it suffices to consider $w\in S_n$ where $\ctype(P(w)^t) \rhd \mu$.
First, consider $w\in S_n$ such that $\ctype(P(w)^t)  = \lambda\unrhd \mu$ where $\lambda$ and $\mu$ differ by moving one box in the Young diagram. 
Assume that we construct $\mu$ from $\lambda$ by taking the last box in some column $j_2$ of $\lambda$ and moving it to the top of column $j_1$, where $j_1< j_2$.

Since $\ctype(P(w)^t) = \ctype(\rev(w))$, we consider the filling $T_{\rev(w)}$ of $\lambda$ that we get from Algorithm \ref{alg: cat insertion}.
Let $a$ be the entry in the last box of column $j_2$. 
Take the box containing $a$ and append it into the end of column $j_1$.

By assumption, we know the resulting shape is $\mu$. Denote the resulting filling of shape $\mu$ by $T_{\rev(w)}^{(\mu)}$. 
We claim that each subword of $\cc(\rev(w))$ that corresponds to a column of $T_{\rev(w)}^{(\mu)}$ is still a valid cocharge word.
The only columns we need to consider are columns $j_1$ and $j_2$. Since the only modification to column $j_2$ is taking off the last entry we added, we see that the resulting word is still a valid cocharge word from Proposition \ref{prop: subword of column is a cocharge word}.

Since $j_1< j_2$ and $T_{\rev(w)}$ is partition shape, there exists an entry $b$ in column $j_1$ which is in the same row as $a$ in $T_{\rev(w)}$. Since $j_1<j_2$, we know we first added $b$, and then $a$, to the filling $T_{\rev(w)}$.  If we let $k_a, k_b$ denote the number of times we read $a,b$ before adding them to $T_w$, we must have $k_b \leq  k_a$. Thus, we know that $\cc(\rev(w))_a \leq  \cc(\rev(w))_b$ from Lemma \ref{lemma:row is sum of cocharge label and mult}. 

Thus adding $a$ to $\cc(\text{rev}(w))^{(j_1)}$ does not introduce a new cocharge value to this word since $\cc(\rev(w))_b$ is in $\cc(\text{rev}(w))^{(j_1)}$.
From case (i) of Lemma \ref{lemma: adding a repeated cocharge value is still a cc word}, we conclude the resulting word is a cocharge word of a permutation. 

If $\ctype(\rev(w))\unrhd \mu$ differ by moving multiple boxes, we repeat this process until we get a filling of $\mu$ where each of the columns correspond to valid cocharge words. It is clear that $\cc(\rev(w))$ is a shuffle of these cocharge words.
Hence $\cw(w) \in D_{\mu^t}$.
From this, we have $\mathcal{C}_\mu = \mathcal{D}_{\mu^t}$. Since $\mathbf{x}^{\mathcal{D}_{\mu^t}}$ is a basis of $R_{\mu}$, we conclude that our set  $\mathbf{x}^{\mathcal{C}_\mu}$  is a basis as well.
\end{proof}

\begin{ex}
We use the same $w, T_w$ as in Example \ref{ex: cat ex}.
For $\mu = (2,2,1,1) \unlhd \ctype(w)$, we can create a new filling $T_w^{\mu}$ of shape $\mu$ by moving the last box in column $2$ of $T_w$ to the end of column 1. This gives us \[T_w^{\mu} = \ytableaushort{6,1,32,54}.\] The two new subwords coming from the columns of $T_w^{\mu}$ are $\cc(w)_1\cc(w)_3\cc(w)_5\cc(w)_6 = 2101$ and 
         $\cc(w)_2\cc(w)_4\cc = 10$, both of which are still cocharge words. 
\end{ex}

\begin{rem}
    Note that \cite[Algorithm 1]{CC} gives a way to take a descent word $\mathbf{a}\in \mathcal{D}_\mu$ and recover an ordered set partition $(A_1|\dots|A_l)$ of $n$ such that $|A_i| = \mu_i$ and $\mathbf{a}\vert_{A_i}\in D_{\mu_i}$. That is: the ordered set partition identifies a canonical way to see $\mathbf{a}$ as a shuffle of descent words. This algorithm does not necessarily agree with our construction of a canonical decomposition of the charge word $\cc(w)$ outlined in Proposition \ref{prop: subword of column is a cocharge word}.
    For example: if we consider $01011\in \mathcal{C}_{2,2,1} = \mathcal{D}_{3,2}$, the Carlsson--Chou construction decomposes this into $\mathbf{01}0\mathbf{1}1$, while our construction decomposes it into $\mathbf{01}01\mathbf{1}$.
\end{rem}

\section{Properties of catabolizability type}\label{section:ctype sum}
Theorem \ref{thm: desc basis are the same} says that two conditions, one involving catabolizability type and the other involving shuffles, are equivalent. In particular, we get the following corollary.

\begin{cor}\label{cor: shuffles of cocharge words have correct ctype}
    Let $u^{(1)}, u^{(2)}, \dots, u^{(l)}$ be permutations where $u^{(j)}$ is of length $\mu_j^t$. Then we have \[\sh(\cc(u^{(1)}),\dots, \cc(u^{(l)})) \subset \{\cc(w) \ | \ w \in S_n, \ \ctype(w)\unrhd \mu\}.\]
\end{cor}

\begin{proof}
This follows from rephrasing the relation $\mathcal{D}_{\mu^t} \subset \mathcal{C}_\mu$ by reversing the words and using Proposition \ref{thm:RSK rev} and \eqref{eq: charge is reverse of cocharge of reverse}.
\end{proof}
However, we can show the following stronger result using properties of Algorithm \ref{alg: cat insertion}.
\begin{thm2}\label{thm: sum of ctypes}

    Let $u^{(1)}, u^{(2)}, \dots, u^{(l)}, w$ be
    permutations such that $\cc(w) \in \sh(\cc(u^{(1)}),\dots, \\ \cc(u^{(l)}))$.
    We have
    \[\ctype(w) \unrhd \ctype(u^{(1)}) + \ctype( u^{(2)}) + \cdots+ \ctype(u^{(l)}),\]
    where $\ctype(u^{(1)}) + \ctype( u^{(2)}) + \cdots +  \ctype(u^{(l)})$ is the partition given by the partwise sum of $\ctype(u^{(1)}),\ctype( u^{(2)}),\dots,\ctype(u^{(l)})$.
    \end{thm2}

We can see that Corollary \ref{cor: shuffles of cocharge words have correct ctype} follows from Theorem \ref{thm: sum of ctypes}. 
For any $u \in S_n$, we know $\ctype(u) \unrhd (1)^n$, since $(1)^{n}$ is the unique smallest partition of size $n$ with respect to dominance order. Hence, for permutations  $u^{(1)}, u^{(2)}, \dots, u^{(l)}, w$  where $u^{(j)}  \in S_{\mu^t_j}$ and $\cc(w) \in \sh(\cc(u^{(1)}),\dots, \cc(u^{(l)}))$, we have
\[\ctype(w) \unrhd \ctype(u^{(1)}) + \ctype( u^{(2)}) + \cdots+ \ctype(u^{(l)}) \unrhd (1)^{\mu^t_1} + (1)^{\mu^t_2} + \cdots + (1)^{\mu^t_j } = \mu.\]
Hence Theorem \ref{thm: sum of ctypes} gives an improved lower bound for the catabolizability type of a shuffle. 

We prove Theorem \ref{thm: sum of ctypes} by introducing a new modification of Algorithm \ref{alg: cat insertion} that keeps track of the difference between $\ctype(w)$ and $\ctype(u^{(1)}) +\cdots +\ctype(u^{(l)})$. 
We first construct our inputs for the algorithm. We restrict to the case where $l = 2$, but the same arguments generalize to larger $l.$
Let $u^{(1)}, u^{(2)},w$ be permutations such that $\cc(w)\in \sh(\cc(u^{(1)}),\cc(u^{(2)}))$. 
If we specify a shuffle of $\cc(u^{(1)})$ and $\cc(u^{(2)})$ that is equal to $\cc(w)$, we have that each letter in $\cc(w)$ corresponds uniquely to a letter in $\cc(u^{(1)})$ or $\cc(u^{(2)})$. Fix such a shuffle. 

We can create fillings $T^{w}_{u^{(1)}},T^w_{u^{(2)}}$ of shapes $\ctype(u^{(1)})$, $\ctype(u^{(2)})$ using Algorithm \ref{alg: cat insertion}. 
Similar to the modification we described in Section \ref{section:cat}, we fill the boxes as we add them to the partition. 
However, we now fill the box with tuples $(k,i)$. 
Consider the box we add to $T^{w}_{u^{(1)}}$ when reading the letter $\ctype(u^{(1)})_{m}$ for the $k$th time. There exists an $i$ such that $\ctype(u^{(1)})_{m}$ corresponds to $\cc(w)_{n-i+1}$ in the fixed shuffle. We fill the box with the tuple $(k,i)$. 

Note that the convention for the index we put in the box is different from what we used in Section \ref{section:descent basis}. 
The $(n-i + 1)$ will be useful when we define a total ordering on these entries.  
We are abusing notation here, since the fillings depend on how we shuffle the two cocharge words to get $\cc(w)$, not just $w$. 

\begin{ex}\label{ex: subfilling}
     Let $w = 5 \ \ 9 \ \ 1 \ \ 2 \ \ 6 \ \ 7 \ \ 10 \ \ 3 \ \ 8 \ \ 4.$
    Then  $\cc(w) = \textcolor{blue}{\textbf{1}} \ \ \textcolor{blue}{\textbf{2}} \ \ \textcolor{blue}{\textbf{0}} \ \ \textcolor{blue}{\textbf{0}} \ \ \textcolor{red}{1} \ \ \textcolor{blue}{\textbf{1}} \ \ \textcolor{blue}{\textbf{2}} \ \ \textcolor{red}{0} \ \ \textcolor{red}{1}\ \ \textcolor{red}{0} \in \sh (\textcolor{blue}{\textbf{120012}},\textcolor{red}{1010})$, where we have colored the entries to denote the shuffle we chose.  
    The resulting fillings  $T^w_{u^{(1)}},T^w_{u^{(2)}}$ are depicted in Figure \ref{fig: T_u1, T_u2}.
    \begin{figure}
        \centering
           \begin{align*} 
    \ytableausetup{textmode, boxframe=normal, boxsize=2em} T^w_{u^{(1)}}= \begin{ytableau}
\footnotesize(3,4) \\ \footnotesize(2,9)\\ \footnotesize(2,5) \\ \footnotesize(1,10)\\\footnotesize(1,7)& \footnotesize(1,8)
\end{ytableau}, \hspace{.2cm}
T^w_{u^{(2)}} = \begin{ytableau}
   \footnotesize(1,2) & \footnotesize(1,6) \\ \footnotesize(1,1) & \footnotesize(1,3)
\end{ytableau}.
    \end{align*}
    \caption{}
    \label{fig: T_u1, T_u2}
    \end{figure}
\end{ex}

From Lemma \ref{lemma:row is sum of cocharge label and mult}, we know if $(k,i)$ appears in row $r$ of $T_{u^{(1)}},$ $T_{u^{(2)}}$ we must have $k + cc(w)_{n-i + 1} = r.$
Furthermore, since the second coordinate of an entry records which letter in $\cc(w)$ created that box, for each $i \in \{1,\dots, n\}$, there exists a unique $k$ such that $(k,i)$ appears in $T^w_{u^{(1)}}$ or $T^w_{u^{(2)}}$. This $(k,i)$ appears exactly once.

We combine $T^w_{u^{(1)}},T^w_{u^{(2)}}$ to get a filling of shape $\ctype(u^{(1)}) + \ctype( u^{(2)})$, denoted  $(T_{u^{(1)}}+T_{u^{(2)}})^w$, by taking the row-wise sum of the two diagrams.
\begin{ex}
    We combine the two diagrams we get from Example \ref{ex: subfilling} to get a filling of $(2+2,1+2,1,1,1)$, depicted in Figure \ref{fig: T_u1+u_2}.
    \begin{figure}
        \centering
           \begin{align*} 
    \ytableausetup{textmode, boxframe=normal, boxsize=2em} (T_{u^{(1)}} + T_{u^{(2)}})^w = \begin{ytableau}
\footnotesize(3,4) \\ \footnotesize(2,9)\\ \footnotesize(2,5) \\ \footnotesize(1,10) &  \footnotesize(1,2) & \footnotesize(1,6)\\\footnotesize(1,7)& \footnotesize(1,8) & \footnotesize(1,1) & \footnotesize(1,3)
\end{ytableau}
    \end{align*}  
            \caption{}
            \label{fig: T_u1+u_2}
    \end{figure}
  
\end{ex}

\vspace{.5cm}
The lexicographic ordering ($\preceq$) on tuples $(a,b) \in \mathbb{Z}^2_{>0}$ is defined by:
\begin{align*}
    (a,b)\prec (c,d) & \Leftrightarrow \begin{cases}
        a < c &  \text{ or } \\ a = c, b <d.
    \end{cases}
\end{align*}
The lexicographic ordering of the tuples keeps track of the order in which we read the entries in Algorithm \ref{alg: cat insertion}. That is: if $(k_1,i_1)\preceq (k_2,i_2)$, then we read $\cc(w)_{n-i_1+1}$ for the $k_1$th time before we read  $\cc(w)_{n-i_2+1}$ for the $k_2$th time.  

Let $T^w_{u^{(1)}+ u^{(2)}}$ be the filling obtained by rearranging the entries within the rows of $(T_{u^{(1)}}+T_{u^{(2)}})^w$ so that the entries within the rows are increasing from left to right with respect to $\preceq$.
We have the following result:
\begin{lemma}\label{lemma:row column increase}
    The filling $T^w_{u^{(1)}+ u^{(2)}}$ satisfies the following conditions:

    \begin{enumerate}[label=(\roman*)]
        \item \label{cond: i} Each $i \in \{1,2,\dots, n\}$ appears in the second coordinate of an entry exactly once 
        \item \label{cond: ii}The entries in each row and column are increasing with respect to the lexicographic ordering on tuples,
        \item \label{cond: iii}For any entry $(k,i)$ in row $r$ of $T^w_{u^{(1)}+ u^{(2)}}$, we have $k + \cc(w)_{n-i+1} = r$.
\end{enumerate}
\end{lemma}
\begin{proof}
    The conditions (i) and (iii) hold by construction. It suffices to show the entries within each columns are increasing. 
    
    Let $(k,i)$ be the entry in the row $r$, column $j$ of $T^w_{u^{(1)}+ u^{(2)}}$, where $r>2$. We know $(k,i)$ is the $j$th smallest entry in row $r$, since we reordered within the rows so that the entries were increasing.
    For each entry in row $r$ of $T^w_{u^{(1)}+ u^{(2)}}$, there is an entry in row $r-1$ that appeared directly below it in $T^w_{u^{(1)}}$ or $T^w_{u^{(2)}}$. Since $T^w_{u^{(1)}}, T^w_{u^{(2)}}$ were constructed through the Algorithm \ref{alg: cat insertion}, we know that their columns are increasing with respect to the lexicographic ordering. 
    
    Thus for each entry $(k',i')$ in row $r$ such that $(k',i')\preceq (k,i)$, there is an entry in the row $(r-1)$ that was directly below $(k',i')$ in $T^w_{u^{(1)}}$ or $T^w_{u^{(2)}}$ which is smaller than $(k',i'$). 
    
    Therefore, if $(k,i)$ is the $j$th smallest entry in row $r$, there are at least $j$ many things in row $(r-1)$ that are smaller than $(k,i)$. Hence $(k,i)$ is larger than the entry directly below it.
\end{proof}

\begin{ex}
    Continuing with our example, we have
     \begin{align*} 
    \ytableausetup{textmode, boxframe=normal, boxsize=2em} T_{u^{(1)}+ u^{(2)}} = \begin{ytableau}
\footnotesize(3,4) \\ \footnotesize(2,9)\\ \footnotesize(2,5) \\   \footnotesize(1,2) & \footnotesize(1,6) & \footnotesize(1,10) \\\footnotesize(1,1) & \footnotesize(1,3)& \footnotesize(1,7)& \footnotesize(1,8) 
\end{ytableau}.
    \end{align*} 
    The entries are increasing within each row and column.
\end{ex}

We now define the notion of row insertion when dealing with these fillings. This is a modification of the classical row insertion used in RSK, except we also check the column increasing condition when inserting and we may modify more than one entry in the row, by sliding part of the row over to the right by one. 

\begin{alg}\label{alg: insertion} (Modified Row Insertion)
Consider a filling $T$ of a two row partition shape, where each entry is a tuple in $\mathbb{Z}^2_{>0}$ and the rows and columns are increasing with respect to the lexicographic ordering. 

Let $(a,b) \in \mathbb{Z}^2_{>0}$. We \textit{insert $(a,b)$} into the second row of this partition by doing the following.
We first find the leftmost $(c,d)$ in the second row that satisfies:
\begin{align}
        (a,b) & \preceq (c,d) \label{conditions: swap1},\\ (a,b) & \succeq (e,f)\label{conditions: swap2}
\end{align}
when $(e,f)$ denotes the entry directly below $(c,d)$.
If no such $(c,d)$ exists, we say $(a,b)$ \textit{pops out} of this row and we are done. 

If such $(c,d)$ exists, replace $(c,d)$ with $(a,b)$. The resulting filling is still increasing in the rows and columns. 
The only place where we could have a contradiction is directly to the left of $(a,b)$, but this cannot happen since $(c,d)$ is the leftmost entry that satisfies \eqref{conditions: swap1} and \eqref{conditions: swap2}.

We now repeat the steps above to insert $(c,d)$ into the same row. We continue until an entry pops out or we insert into the last box of the row. 

If we insert into the last box of the row, replacing $(g,h)$, we finish our insertion process by appending $(g,h)$ to the end of the second row if the resulting filling has increasing columns and is still a partition shape. This only happens if the second row was strictly shorter than the first. In this case, nothing pops out.
Otherwise, $(g,h)$ pops out and the insertion is complete. 

Note that once we successfully insert $(a,b)$, the entry $(c,d)$ gets inserted to the direct right of $(a,b)$ or it pops out. 
From this, we can see that this process is equivalent to determining where we can put $(a,b)$ and then moving the entries larger than it to the right by one until we find a contradiction with the bottom row or we reach the end of the row. 
\end{alg}

We illustrate this through the following examples. The first one results in an entry popping out. 
\begin{ex}\label{ex: row 2 insert 1}
We insert $(a,b) = (2,6)$ into the second row of the following partition:
    \begin{align*} 
    \ytableausetup{textmode, boxframe=normal, boxsize=2em} \begin{ytableau}
\footnotesize(2,4) & \footnotesize(3,2) & \footnotesize(3,3) & \footnotesize(4,1) \\\footnotesize(1,1) & \footnotesize(2,5) & \footnotesize(2,7) & \footnotesize(3,8) & \footnotesize(5,2)
\end{ytableau}
\end{align*}
Since $(3,2)$ is the leftmost entry satisfying \eqref{conditions: swap1} and \eqref{conditions: swap2}, we replace $(3,2)$ with $(2,6)$:
  \begin{align*} 
    \ytableausetup{textmode, boxframe=normal, boxsize=2em} \begin{ytableau}
   \none & \none[(3,2)] \\
\footnotesize(2,4) & \footnotesize\textbf{(2,6)} & \footnotesize(3,3) & \footnotesize(4,1) \\\footnotesize(1,1) & \footnotesize(2,5) & \footnotesize(2,7) & \footnotesize(3,8) &\footnotesize(5,2)
\end{ytableau}
\end{align*}
Now we repeat: we replace $(3,3)$ with $(3,2)$:
  \begin{align*} 
    \ytableausetup{textmode, boxframe=normal, boxsize=2em} \begin{ytableau}
   \none &\none & \none[(3,3)] \\
\footnotesize(2,4) & \footnotesize(2,6) & \footnotesize\textbf{(3,2)} & \footnotesize(4,1) \\\footnotesize(1,1) & \footnotesize(2,5) & \footnotesize(2,7) & \footnotesize(3,8) &\footnotesize(5,2)
\end{ytableau}
\end{align*}
However, we cannot replace $(4,1)$ with $(3,3)$ since $(3,8) \not\preceq (3,3)$. Hence $(3,3)$ pops out of this insertion and we are done.
\end{ex}

We also give an example of when nothing pops out of the insertion process. 
\begin{ex}\label{ex: row 2 insert 2}
    We insert $(a,b) = (2,6)$ into the second row of the following partition:
    \begin{align*} 
    \ytableausetup{textmode, boxframe=normal, boxsize=2em} \begin{ytableau}
\footnotesize(2,4) & \footnotesize(3,2) \\\footnotesize(1,1) & \footnotesize(2,5) & \footnotesize(3,1) 
\end{ytableau}
\end{align*}
Since $(3,2)$ is the leftmost entry satisfying \eqref{conditions: swap1} and \eqref{conditions: swap2}, we replace $(3,2)$ with $(2,6)$:
  \begin{align*} 
    \ytableausetup{textmode, boxframe=normal, boxsize=2em} \begin{ytableau}
   \none & \none[(3,2)] \\
\footnotesize(2,4) & \footnotesize\textbf{(2,6)} \\\footnotesize(1,1) & \footnotesize(2,5) & \footnotesize(3,1) 
\end{ytableau}
\end{align*}
Now, we can append $(3,2)$ to the end of row 2 while maintaining the column increasing condition.
Adding a box to the second row does not change the fact that the shape is a two row partition. Hence we insert $(3,2)$ into the empty spot. 
 \begin{align*}
    \ytableausetup{textmode, boxframe=normal, boxsize=2em} \begin{ytableau}
\footnotesize(2,4) & \footnotesize(2,6) & \footnotesize\textbf{(3,2)} \\\footnotesize(1,1) & \footnotesize(2,5) & \footnotesize(3,1) 
\end{ytableau}
\end{align*}
Note that nothing pops out at the end of this insertion.
\end{ex}

We make the following observation.
\begin{lemma}\label{lemma: insertion boosts}
 If we insert $(a,b)$ into row 2 of $T$ and $(x,y)$ pops out, we have $(a,b)\preceq (x,y)$.
\end{lemma}
\begin{proof}
    If we are unable to insert $(a,b)$, then $(x,y) = (a,b)$. Otherwise, the insertion replaces larger entries with smaller ones. 
\end{proof}

We can extend this definition of row insertion to any partition shape tableau $T$ filled with tuples in $\mathbb{Z}^2_{>0}$ where the rows and columns are increasing. For any row $r>1$ of $T$ and $(a,b)\in \mathbb{Z}^2_{>0}$, we define $T\leftarrow_{r} (a,b)$ to be the result of inserting $(a,b)$ into row $r$ of $T$. Unlike RSK, we only modify row $r$; we do not continue inserting into higher rows. It is a simple check to see that the resulting filling $T\leftarrow_{r} (a,b)$ is still increasing in the rows and columns.

\begin{lemma}\label{lemma: inserting larger things higher}
    Consider partition shape tableau $T$ filled with tuples in $\mathbb{Z}^2_{>0}$ where the entries within the rows and columns are increasing. Consider two tuples $(a,b) \preceq (a',b')$.
    Let $T\leftarrow_r (a,b)$ denote the tableau we get by inserting $(a,b)$ into row $r$ of $T$, where $r$ is not the top row. 
    \begin{enumerate}
        \item Assume $(x,y)$ pops out when we insert $(a,b)$ into row $r$ of T. If $(x',y')$ pops out when we insert $(a',b')$ into row $(r+1)$ of $T\leftarrow_r (a,b)$, then $(x,y)\preceq (x',y')$.
        \item If nothing pops out when we insert $(a,b)$ into row $r$ of $T$, then nothing pops out when we insert $(a',b')$ into row $(r+1)$ of $T\leftarrow_r (a,b)$.
        \end{enumerate}
\end{lemma}

\begin{proof}

(1) Note that if $(x,y) = (a,b)$, the claim follows immediately.
Otherwise, let $(p,q)$ denote the entry in row $(r+1)$ directly above $(x,y)$ in $T$. Note that the entry $(p,q)$ may not exist: in that case, assume $(p,q)$ to be the empty box at the end of row $(r+1$).
Let $(g,h)$ denote the entry in row $(r+1)$ directly above $(a,b)$ in $T\leftarrow_r (a,b)$.

We can see that rows $r,r+1$ of $T$ look like the following:
  \ytableausetup{mathmode, boxframe=normal, boxsize=2.75em}
\begin{align*}
  T:   \hspace{1cm} \begin{ytableau}
*(white) & \footnotesize(g,h) & *(white)\cdots & *(white)  \footnotesize(p,q) & *(white)\\  
*(white) & *(gray) & *(gray)\cdots & \footnotesize(x,y) & *(white)
\end{ytableau}.
\end{align*}

We can visualize $T\leftarrow_r (a,b)$ in the following way. 
First, we move the shaded boxes in row $r$, which are the boxes from the one directly below $(g,h)$ up to (but not including) $(x,y)$, over one to the right. Then we put $(a,b)$ directly below $(g,h)$ which pops out $(x,y)$. 

Rows $r,r+1$ of $T\leftarrow_r (a,b)$ look like the following:
\begin{align*}
   T \leftarrow_r (a,b): \hspace{1cm} \begin{ytableau}
*(white) & \footnotesize(g,h) & *(white)\cdots  & *(white)\footnotesize(p,q) & *(white)\\  
*(white) & \footnotesize(a,b) & *(gray)\cdots &  *(gray) & *(white)
\end{ytableau}.
\end{align*}

Now, we insert $(a',b')\succeq (a,b)$ into row $(r+1)$ of  $T\leftarrow_r (a,b)$. Let $(x',y')$ be the entry that pops out of this insertion, if such entry exists. To show $(x,y)\preceq (x',y')$, we use the fact that $(x,y)\preceq (p,q)$. It suffices to show $(p,q)\preceq (x',y')$.

 We proceed by cases. First, note that if $(p,q)\preceq (a',b') $, the claim automatically holds since $(a',b')\preceq (x',y')$ (with potential equality).

The second case we consider is when $(g,h) \preceq (a',b')\preceq (p,q)$. In this case, we can always insert $(a',b')$ between $(g,h)$ and $(p,q)$.
Consider $ (g_1,h_1) \preceq (g_2,h_2)$ that appear consecutively in row $r$ of $T\leftarrow_r (a,b)$ such that $(g,h) \preceq (g_1,h_1) \preceq (a',b') \preceq (g_2,h_2)\preceq (p,q)$. This is equivalent to saying that in $T\leftarrow_r (a,b)$, we have the following configuration:
 \ytableausetup{mathmode, boxframe=normal, boxsize=3em}
\begin{align*}
   T \leftarrow_r (a,b): \hspace{1cm}  \begin{ytableau}
*(white) & \footnotesize(g,h) & *(white)\cdots & \footnotesize(g_1,h_1) &  *(yellow)\footnotesize(g_2,h_2) & *(yellow)  \cdots & *(white)  \footnotesize(p,q) & *(white)\\  
*(white) & \footnotesize(a,b) & *(gray)\cdots & *(gray) & *(gray)\footnotesize(s_1,t_1)& *(gray)\footnotesize(s_2,t_2)&*(gray)  & *(white)
\end{ytableau}.
\end{align*}

Since we know that all the shaded entries in row $r$ were shifted 1 to the right when inserting $(a,b)$ into row $r$ of $T$.we know that the entry $(s_1,t_1)$ that appears directly below $(g_2,h_2)$ was originally below $(g_1,h_1)$ in $T$.
Thus $(s_1,t_1) \preceq (g_1,h_1) \preceq (a',b') \preceq (g_2,h_2)$, which means we can replace $(g_2,h_2)$ with $(a',b')$.




We can also move all the entries from $(g_2,h_2)$ up to (but not including) $(p,q)$ over to the right by one, because this is just equivalent to realigning rows $r$ and $(r+1)$ of $T\leftarrow_r (a,b)$ to how they were in $T$.
In particular, we can place $(g_2,h_2)$ on top of $(s_2,t_2)$, since we know $(s_2,t_2)\preceq (g_2,h_2)$.
The result of continuing the insertion up to $(p,q)$ is the following configuration:
\begin{align*}
    \begin{ytableau}
    \none & \none & \none & \none & \none & \none &  \none[(p,q)]\\  
*(white) & \footnotesize(g,h) & *(white)\cdots & \footnotesize(g_1,h_1) & \footnotesize(a',b') &  *(yellow)\footnotesize(g_2,h_2) & *(yellow)  \cdots  & *(white)\\  
*(white) & \footnotesize(a,b) & *(gray)\cdots & *(gray) & *(gray)\footnotesize(s_1,t_1)& *(gray)\footnotesize(s_2,t_2)&*(gray)  & *(white)
\end{ytableau}.
\end{align*}

At this point, either we are able to continue by inserting $(p,q)$ into row $(r+1)$, or $(p,q)$ pops out.
From this, we know the smallest possible thing that could pop out is $(p,q)$. Thus $(p,q)\preceq (x',y')$.  

Now, we consider the case where $(a',b')\preceq (g,h)$. 
We know $(a',b')$ can be inserted in some spot to the left of $(g,h)$ by just checking \eqref{conditions: swap1}: the column condition follows automatically since $(a,b) \preceq (a',b')$, thus $(a',b')$ is larger than any entry to the left of $(a,b)$ in row $r$. Let $(c,d)$ be the leftmost entry in row $(r+1)$ such that $(a',b')\preceq (c,d)$.

\begin{align*}
      T \leftarrow_r (a,b): \hspace{1cm}    \begin{ytableau} 
*(white) &  *(yellow)\footnotesize{(c,d)} & *(yellow)\dots & \footnotesize(g,h) &  \cdots & (p,q) & *(white)\\*(white) &  *(white) &   
*(white) & \footnotesize(a,b) & *(gray)\cdots & *(gray) & *(white)
\end{ytableau}.
\end{align*}

We replace $(c,d)$ with $(a',b')$. Now, any entry between $(c,d)$ and $(g,h)$ (which correspond to the boxes colored yellow), must also be larger than $(a,b)$, so we can keep inserting until we have replaced $(g,h)$. This gives us the following configuration:
\begin{align*}
    \begin{ytableau} 
    \none & \none & \none & \none[(g,h)] \\
*(white) & \footnotesize{(a',b')} & *(yellow)\dots & *(yellow) &  \cdots & (p,q) & *(white)\\*(white) &  *(white) &   
*(white) & \footnotesize(a,b) & *(gray)\cdots & *(gray) & *(white)
\end{ytableau}.
\end{align*}

The next step is to insert $(g,h)$ into row $(r+1)$: we can see that this is exactly the second case we considered. Thus, from the argument above, we know the insertion process will continue until we reach $(p,q)$ and  the claim holds. 

(2) The same argument holds: if such $(x,y)$ does not exist, then $(p,q)$ does not exist. Thus the two cases we consider are $(a',b')\preceq (g,h)$ and $(g,h) \preceq (a',b')$, which correspond to the first two cases in the previous argument. We can keep inserting into row $(r+1)$ until we add a box at the end of row $(r+1)$, which we know we can do since row $r$ of $T\leftarrow_r (a,b)$ is one longer than row $r$ of $T$.  
\end{proof}

Now, we use this new notion of insertion to define a modification of Algorithm \ref{alg: cat insertion} that keeps track of the difference between $\ctype(w)$ and $\ctype(u^{(1)})+ \cdots + \ctype(u^{(l)})$. 
The input $(\cc(w), \emptyset, T^w_{u^{(1)}, \dots, u^{(l)}})$ depends on the shuffle of $\cc(u^{(1)}),\dots, \cc( u^{(l)})$ that is equal to $\cc(w)$ that we fix. Throughout this algorithm, we can think of the filling $T$ in the third coordinate of our tuple as keeping track of the ``lower bounds'' of insertion: that is, if we read a certain entry in $T$, we must add a box there. However, we may add boxes earlier, which is when we perform the chains of insertions.  The final result in the third coordinate will be a tableau $T_w$ with $\shape(T_w) = \ctype(w)$.

We include an example of this algorithm in the Appendix.

\begin{alg}\label{alg:chains}
    The input is $(\cc(w),\emptyset, S)$, where $S$ is a filling of a partition shape that satisfies the conditions \ref{cond: i}  -- \ref{cond: iii} in Lemma \ref{lemma:row column increase}. For example, we can take $S$ to be $T^w_{u^{(1)}, \dots, u^{(l)}}$.
    
    As we did for Algorithm \ref{alg: cat insertion}, we apply the function $f$ repeatedly to the first two coordinates until we get $(\emptyset, \ctype(w))$ in the first two coordinates. We modify the shape of the filling in the last coordinate whenever we add a box to the second coordinate, so that the new shape is larger in dominance order.

    At each step, we say we \textit{read $(k,i)$} if we read the letter corresponding to $\cc(w)_{n-i+1}$ for the $k$th time.
    After each step of the algorithm, we get a triple $(z, \nu , T)$ where $z$ is a word, $\nu$ is some partition, and $T$ is a tableau of partition shape.

    Consider the step when we read $(k,i)$, with  input $(z,\nu,T)$, which results in adding a box to row $r$ of $\nu$. Let $(\tilde{z}, \nu + \epsilon_{r})$ be the new first and second coordinate of our tuple. 
    
    By condition \ref{cond: i}, we know there exists a $(m,i)$ that appears in $T$.  If $k=m$, we do not modify the filling $T$ and the new tuple we consider is $(\tilde{z},\nu+\epsilon_{r}, T)$.
    Note that by condition \ref{cond: iii}, we know that $(k,i)$ appears in row $r$ of $T$ (To see an example, see \eqref{appendix: ex1} in the appendix).

    However, if $(m,i)$ is in row $r'$ where $r<r'$, we will create a new tableau $\tilde{T}$ of partition shape in the following way. We first look at $T$, where $\star$ denotes the box containing $(m,i)$:
    \begin{align*}
        \ytableausetup{mathmode,smalltableaux}
     \begin{ytableau}
      \none & \none & *(white) & *(white)& *(white) \\ \none & \none & *(white) & *(white)& *(white) \\\none &  \none & *(white) & *(white)& *(white)& *(white) \\ \none[\footnotesize{\text{row r'}}] & \none & *(white) & *(white)\star & *(white)& *(white)& \\ \none & \none & *(white) & *(white) & *(white)& *(white)& *(white)& *(white)& *(white)\\ \none[\footnotesize{\text{row r}}] & \none &*(white) & *(white)& *(white)& *(white)& *(white)& *(white)& *(white) \\ \none & \none & *(white) & *(white)& *(white)& *(white)& *(white)& *(white)& *(white)& *(white)& *(white) \\ 
    \end{ytableau}.
    \end{align*}

    Delete the box in $T$ containing $(m,i)$ as well as the ones directly above it in $T$.
    This results in deleting entries $\{(m,i),(m_1,i_1),\dots, (m_s,i_s)\}$, where $(m_t,i_t)$ denotes the entry that was $t$ boxes above $(m,i)$ and $s$ is the total number of boxes that appear above $(m,i)$ in $T$. 
    The resulting diagram looks like this:

        \begin{align*}
        \ytableausetup{mathmode,smalltableaux}
        \begin{ytableau}
        \none & \none & *(white) & \none & *(white)\\
        \none & \none & *(white) & \none & *(white)\\ \none &  \none & *(white) & \none[] & *(white)& *(white) \\ \none[\footnotesize{\text{row r'}}] & \none & *(white) & \none  & *(white)& *(white)& \\ \none & \none & *(white) & *(white) & *(white)& *(white)& *(white)& *(white)& *(white)\\ \none[\footnotesize{\text{row r}}] & \none &*(white) & *(white)& *(white)& *(white)& *(white)& *(white)& *(white) \\ \none & \none & *(white) & *(white)& *(white)& *(white)& *(white)& *(white)& *(white)& *(white)& *(white) \\ 
    \end{ytableau} \hspace{1cm} \text{ removed entries: }\{(m,s), (m_1,i_1), (m_2,i_2), (m_3,i_3)\}
    \end{align*}

    Note that since the entries are increasing within the columns of $T$, we have
    \begin{align}\label{eq: insertion is increasing1}
        (m,i) \preceq (m_1,i_1)\preceq \cdots \preceq (m_s,i_s)
    \end{align}
 
    The rows and columns are still increasing with respect to the lexicographic ordering if we ignore the empty spaces in the column that contained $\{(m,i),(m_1,i_1),\dots, (m_s,i_s)\}$. 
    
    We will create a new partition shape filling of $n$ boxes by inserting the entries we removed back into this filling.
    
   Let $d = r'-r$.  We first insert $(m - d, i)= (k,i)$ into row $r$ using Algorithm \ref{alg: insertion}.  Note that we use $(m-d,i)$, rather than $(m,i)$, to maintain condition \ref{cond: iii}. This change does not change the fact that the entries in the filling satisfy condition \ref{cond: i}. 

    If an entry pops out, we denote that entry by $(x_0,y_0)$. Now, we continue this process, inserting $(m_1-d, i_1)$ into row $r+1$.
  If some entry $(x_1,y_1)$ pops out, we know $(x_0,y_0)\preceq (x_1,y_1)$ from Lemma \ref{lemma: inserting larger things higher} (1) since $(m-d,i)\preceq (m_1-d, i_1)$.
      
     In this way, we insert $(m_1-d,i_1,),\dots, (m_s-d, i_s)$ into rows $r+1,\dots, r+s$.

    \begin{align*}
        \begin{ytableau}
        \none & \none & \none & \none & *(white) & \none & *(white)\\ 
        \none & \none & \none & \none & *(white) & \none & *(white)\\  \none &\none[\scriptscriptstyle (m_3-d,i_3)\to] & \none  & \none & *(white) & \none & *(white)& *(white) \\   \none &\none[\scriptscriptstyle (m_2-d,i_2)\to] & \none  & \none  & *(white) & \none  & *(white)& *(white)& \\ \none &\none[\scriptscriptstyle (m_1-d,i_1)\to] & \none  & \none  & *(white) & *(white) & *(white)& *(white)& *(white)& *(white)& *(white)\\  \none &\none[\scriptscriptstyle (m-d,i)\to] & \none  & \none &*(white) & *(white)& *(white)& *(white)& *(white)& *(white)& *(white) \\ \none & \none & \none & \none & *(white) & *(white)& *(white)& *(white)& *(white)& *(white)& *(white)& *(white)& *(white) \\ 
    \end{ytableau} .
    \end{align*}

     If we insert into a row with a gap in the middle, we ensure that we maintain the empty spot by skipping over it in the insertion process. That is: we never insert in the empty spot in the middle of the row. 

 By Lemma \ref{lemma: inserting larger things higher} (2) and \eqref{eq: insertion is increasing1}, we know that once we have an insertion where nothing pops out, we will never have anything pop out for the later insertions as well.  Let $j\leq s$ be the index such that inserting $(m_j -d, i_j)$ is the last insertion where some entry $(x_j, y_j)$ pops out.
The elements that pop out at each step have the following relation:
    \begin{align}\label{ineq: chains}
        (k,i)\preceq (x_0,y_0)\preceq (x_1,y_1)\preceq \dots \preceq (x_{j},y_{j}),
    \end{align}
where either $j =s$ or nothing popped out when inserting $(m_{j+1} -d, i_{j+1})$ into row $r+j+1$.

For example, if we assume $j = 2$ in our example, we get:

 \begin{align*}
        \begin{ytableau}
        *(white) & \none & *(white)\\ 
         *(white) & \none & *(white) & \none & \none & \none & \none \\   *(white) & \none & *(white)& *(white) & *(white) & \none  & \none \\  *(white) & \none  & *(white)& *(white)&  *(white)& \none &  \none & \none  & \none & \none & \none &\none[\scriptscriptstyle \to (x_2,y_2)]  \\  *(white) & *(white) & *(white)& *(white)& *(white)& *(white)& *(white) & \none & \none  & \none & \none  &\none[\scriptscriptstyle \to (x_1,y_1)] \\   *(white) & *(white) & *(white)& *(white)& *(white)& *(white)& *(white) & \none & \none  & \none & \none  &\none[\scriptscriptstyle \to (x_0,y_0)]  \\ *(white) & *(white)& *(white)& *(white)& *(white)& *(white)& *(white)& *(white)& *(white)
    \end{ytableau} .
    \end{align*}

Once we have inserted up to $(m_s-d, i_s)$, we still have an empty column where $(m,i)$ used to be and above. 

Since the insertion algorithm replaces entries with smaller ones, we know that the entry below the empty space in row $r'$ must still be smaller than $(m,i)$.
If there exists an entry $(x_0,y_0)$ that popped out of row $r$, we have the following inequalities using \eqref{ineq: chains}:
    \begin{align}\label{ineq: chains2}
        (k+d,i) = (m,i)\preceq (x_0 +d,y_0)\preceq (x_1 + d,y_1)\preceq \dots \preceq (x_j+d ,y_j).
    \end{align}
Thus we can fill the empty column with the entries $(x_0 +d,y_0),(x_1 + d,y_1), (x_j+d ,y_j)$ and preserve the column increasing condition for all the columns.  

 \begin{align*}
        \begin{ytableau}
        *(white) & \none& *(white)\\ 
         *(white) & \circ & *(white)  & \none & \none & \none \\   *(white) & \diamond  & *(white)& *(white) & *(white)  & \none  & \none \\  *(white) & \ast  & *(white)& *(white)& *(white) & \none  \\  *(white) & *(white) & *(white)& *(white)& *(white)& *(white)& *(white)\\ *(white) & *(white)& *(white)& *(white)& *(white)& *(white)& *(white)   \\ *(white) & *(white)& *(white)& *(white)& *(white)& *(white)& *(white)& *(white)& *(white) \\ 
    \end{ytableau} \text{ where } \ast = (x_0,y_0), \diamond = (x_1,y_1), \circ = (x_2,y_2).
    \end{align*}

Now that we have put all the entries back into the filling, we rearrange the entries within rows $r'$ and above so that the entries within the rows are increasing and we have a partition shape with no gaps.
Let the resulting diagram be $\tilde{T}$. We can use the same argument used in the proof of Lemma \ref{lemma:row column increase}(ii) to show that the $\tilde{T}$ is increasing within the rows and columns, since the filling we had before rearranging within the rows was column strict. Using this, we know that $\tilde{T}$ still satisfies conditions \ref{cond: i}-- \ref{cond: iii}.
In this case, the new tuple we consider is $(\tilde{z},\nu+\epsilon_{r}, \tilde{T})$.

We proceed by reading the next letter in $\tilde{z}$. We continue this process until the first coordinate of our tuple is the empty word.
    \end{alg}

Algorithm \ref{alg:chains} satisfies the following properties.  When we are reading an entry $(k,i)$, we say that smaller than $(k,i)$ in the lexicographic order has \textit{already been read}. 
    \begin{lemma}\label{lemma: props of the alg}
    Assume we apply Algorithm \ref{alg:chains} with intial input $(\cc(w), \emptyset, S)$.
        Each intermediate tuple $(z, \nu, T)$ must  satisfy the following conditions:
    \begin{enumerate}[label=(\alph*)]
        \item As partitions, we have $\shape(T)\unrhd \shape(S)$, where $S$ is the filling in the original input.
        \item A tuple $(k,i)$ that appears in $T$ has already been read if and only if it is contained in the subdiagram of $T$ of shape $\nu$.
        \item A tuple $(k,i)$ that appears in $T$ has not been read if and only if the letter corresponding to $\cc(w)_{n-i+1}$ has not been deleted from the input word. 
    \end{enumerate}
    \end{lemma}
    \begin{proof}
    Note that by assumption, we have that $T$ satisfies the conditions \ref{cond: i}--\ref{cond: iii}. 
    We proceed by induction. Note that the statement is true for the initial filling $S$ before we read $(1,1)$. We proceed by cases.

    Assume we read $(k,i)$ but cannot add a box to $\nu$. By \ref{cond: iii}, we know that there exists a tuple $(m,i)\in T$ . By induction, we know that $(m,i)$ cannot have been read yet by (c), since it corresponds to an entry that has not been deleted from the input word. Thus $m\geq k$.
    Furthermore, if $m=k$, we would have that the box below $(k,i)$ in $T$ must be contained in the subdiagram of shape $\nu$ by (b) and \ref{cond: ii}, since it must have already been read. However, this would imply that we would be able to add a box to the diagram $\nu$, which is a contradiction. Thus $k<m$. Since this does not affect any of the conditions, we have that (a)--(c) are true for this tuple.

    Now we consider the case where reading $(k,i)$ results in adding a box to row $r$ of $\nu$ but does not change the filling $T$.
    This does not affect condition (a).     
    In this case, we know $(k,i)$ appears in row $r$ of $T$. Since we assume that $T$ satisfied conditions (a)--(c) before we read $T$, we know that $(k,i)$ is the smallest entry in row $r$ of $T$ that is not contained in the subdiagram of shape $\nu$. This implies that the box in $T$ containing $(k,i)$ is exactly the new box in $\nu + \epsilon_r$, thus (b) and (c) hold as well. 

    Finally, we consider the case where reading $(k,i)$ results in creating a new filling $\tilde{T}$. 

    In this case, we know that $(m,i)$ appears in row $r'>r$ of $T$.
    This means we delete $(m,i)$ (and all the entries above it) and then insert $(k,i)$ into row $r$. Note that from (b) on $T$, we know that none of the entries we delete can be contained in the subdiagram of $T$ of shape $\nu$, since all the entries we delete are larger than $(m,i)$. 

   Let $(x,y)$ be the smallest, unread entry in row $r$ of $T$. From condition (c) and the fact that reading $(k,i)$ results in adding a box to row $r$ of $\nu$, it follows that the box under $(x,y)$ must be contained in the subdiagram of shape $\nu$. Thus the entry below $(x,y)$ has already been read. Furthermore, since $(k,i)$ is the smallest unread entry, we have that $(k,i)\preceq (x,y)$. Thus inserting $(k,i)$ into row $r$ results in replacing $(x,y)$ with $(k,i)$. Note that this box in $\tilde{T}$ is in the same position as the new box we add to $\nu$. 

   Note that from equations \eqref{eq: insertion is increasing1}, we know the insertion process only involves entries that are larger than $(k,i)$, thus the entries in the subdiagram of shape $\nu$ are fixed. 
   Furthermore, we can see that when we rearrange the rows after the insertion so that the final shape is partition, we do not touch any of the entries in the subdiagram of shape $\nu+\epsilon_r$. Hence condition (b), (c) hold.

Finally, we can see this insertion process corresponds to moving entries in $T$ to lower rows. Thus we have condition (1).


    \end{proof}

\begin{cor}\label{cor: output fillling}
   When applying Algorithm to $(\cc(w), \emptyset, S)$, we have that the final output is of the form $(\emptyset, \ctype(w), T_w) $ where  $\shape(T_w) = \ctype(w)$.
\end{cor}
\begin{proof}
We know that the final tuple $(\emptyset, \ctype(w), T_w) $ must satisfy Lemma \ref{lemma: props of the alg} (a) --(c).
By (c), we know that all the entries in $T_w$ must be read. Then, by condition (2), we have $\ctype(w) = \shape(T_w)$.
\end{proof}

Theorem \ref{thm: sum of ctypes} immediately follows from this algorithm.
\begin{proof}[Proof of Theorem \ref{thm: sum of ctypes}] 
We apply Algorithm \ref{alg:chains} to the initial input $(\cc(w), \emptyset, T^w_{u^{(1)}, \dots, u^{(l)}})$. 
From Lemma \ref{lemma: props of the alg}, at any point in the algorithm,  where the tuple is of the form $(z,\nu,T)$, we have  \[\ctype(u^{(1)}) + \cdots +\ctype(u^{(l)})  = \shape (T^w_{u^{(1)},\dots,  u^{(l)}}) \unlhd \ \shape(T).\] From Corollary \ref{cor: output fillling}, it follows that $\ctype(u^{(1)}) + \cdots+ \ctype( u^{(l)})\unlhd \shape (T_w) = \ctype(w)$.    
\end{proof}

    We can also use this modified algorithm to prove the following statement about how certain modifications to permutations raise the catabolizability type.
    
\begin{prop}\label{prop: swaps raise ctype}
   Consider $w\in S_n$ and index $i$ such that $w_i  +1 < w_{i+1}$. If $\tilde{w} = w_1\dots w_{i+1}w_i\dots w_n$, we have $\ctype(\tilde{w}) \unrhd \ctype(w)$.
\end{prop}
\begin{proof}
    We know  $\cc(w)_i < \cc(w)_{i+1}$, since they differ by more than one. This swap does not change the cocharge label of any of the letters. 
    Thus $\cc(\tilde{w})$ is the word we get by swapping $\cc(w)_i$ and $\cc(w)_{i+1}$ in $\cc(w)$.

    Let $T_w$ be the filling of $\ctype(w)$ we get from Algorithm \ref{alg: cat insertion}, where we fill a box with tuple $(k,i)$ if we add the box when reading the letter $\cc(w)_{n-i + 1}$ for the $k$th time. There exist unique $k_1,k_2$ such that $(k_1,n-i+1), (k_2,n-i)$ appear in $T_w$. Replace these two tuples with  $(k_1,n-i), (k_2,n-i+1)$ to account for the swapping of $\cc(w)_i$ and $\cc(w)_{i+1}$. Denote the resulting filling by $T_{w}^{\tilde{w}}$.

    The filling $T_{w}^{\tilde{w}}$ satisfies the conditions (1), (3) in Lemma \ref{lemma:row column increase} by construction. We can check that it satisfies (2) as well. 
    Note that the only case where swapping the second coordinates changes the relative order on the elements in $T_{w}^{\tilde{w}}$ is if $k_1 = k_2$. 
    In this case, the potential issue occurs if $(k_1,n-i), (k_2,n-i+1)$ appear in the same row or column of $T_w$.
    However, this never happens. Since $\cc(w)_i \neq \cc(w)_{i+1}$, if the two appeared in the same row we would have $k_1 > k_2$ from Lemma \ref{lemma:row is sum of cocharge label and mult}.
    If $k_1 = k_2$, then $(k_1,n-i)$ appears in a higher row than $ (k_2,n-i+1)$. Since we add $(k_1,n-i)$ to $T_w$ before we add $(k_2,n-i+1)$, it is impossible that these two tuples appear in the same column of $T_w$. Hence $T_{w}^{\tilde{w}}$ satisfies (2).

    Now, we apply Algorithm \ref{alg:chains} to the initial input $(\cc(\tilde{w})), \emptyset, T_{w}^{\tilde{w}})$. 
    Since $\text{shape}(T_{w}^{\tilde{w}}) = \ctype(w)$ and the resulting tuple is $(\emptyset, \ctype(\tilde{w}), T_{\tilde{w}})$ with $\shape (T_{\tilde{w}}) = \ctype(\tilde{w})$, it follows that $\ctype(\tilde{w}) \unrhd \ctype(w)$. 
\end{proof}

\section{Antisymmetric part}\label{sec: remarks}
In the work of Garsia--Procesi \cite[Section 3]{GP}, the ungraded Frobenius character of $R_\mu$ is identified to be the complete homogeneous symmetric function $h_\mu$ directly from the structure of $R_\mu$. 
However, they require further properties of $R_\mu$, as well as the modified Hall-Littlewood polynomials, to identify the graded Frobenius character. 

In this section, we use properties of our charge monomial basis $\{\mathbf{x}^{\cw(w)} \ | \ \ctype(P(w)^t)\unrhd \mu\}$ to give a direct, elementary proof of the fact that $\Frobq (R_{\mu}) = \mHL_{\mu}[X;q]$ that only depends on the well-known combinatorial formula \eqref{eq:catabolizable HL formula} and the ungraded character.

We do this by showing that for any composition $\gamma = (\gamma_1,\dots, \gamma_l)\models n$, we have $\Hilb(N_\gamma R_{\mu}) = \langle e_\gamma , \mHL_{\mu}[X;q]\rangle$, where $N_\gamma$ is the antisymmetrizer with respect to the Young subgroup $S_\gamma = S_{\gamma_1}\times \cdots\times  S_{\gamma_l} \subset S_n$:  
\begin{align*}
    N_\gamma = \sum\limits_{\sigma\in S_\gamma} \text{sgn}(\sigma)\sigma.
\end{align*}

Recall that we have the following formula for  $\mHL_{\mu}[X;q]$ for $\mu\vdash n$, using Proposition \ref{thm: charge cocharge trabspose}:
\begin{align}\label{eq: cat HL with transpose}
\mHL_{\mu}[X;q] =  \sum\limits_{\substack{S\in \syt\\  \ctype(S^t)\unrhd\mu}}  q^{\charge(S)} s_{\shape(S^t)}.
\end{align}
We apply $\omega$ to both terms. Since $\{m_\lambda\}, \{h_\gamma\}$ are dual with respect to the inner product, we get
\begin{align*}
     \langle e_{\gamma},\  \mHL_{\mu}  [X;q] \rangle =  \langle h_{\gamma},\  \omega \mHL_{\mu} [X;q] \rangle = \langle m_\gamma\rangle  \ \omega \mHL_{\mu} [X;q] , 
\end{align*}
where $\langle m_\gamma\rangle  \ \omega \mHL_{\mu} [X;q] $ denotes the coefficient of $m_\gamma$ in the monomial symmetric function expansion of $\omega \mHL_{\mu} [X;q] $.
We compute $\omega \mHL_{\mu} [X;q] $ using $\omega s_\lambda = s_{\lambda^t}$ for any partition $\lambda$ and the combinatorial formula \eqref{eq: cat HL with transpose}:
\begin{align*}
    \omega \mHL_{\mu} [X;q]   &= \sum\limits_{\substack{S\in \syt\\  \ctype(S^t)\unrhd\mu}}  q^{\charge(S)}s_{\shape(S)}  .
\end{align*}

From the above, we can see
\begin{align*}
    \langle e_{\gamma},\  \mHL_{\mu}  [X;q] \rangle &=  \sum\limits_{\substack{S \in \syt\\  \ctype(S^t)\unrhd\mu}}  q^{\charge(S)} \left(\langle m_\gamma \rangle \   s_{\shape(S)} \right) \\&= \sum\limits_{\substack{S\in \syt\\  \ctype(S^t)\unrhd\mu}}  q^{\charge(S)}   K_{\shape(S),\gamma}.
\end{align*}

Furthermore, we can rewrite expression \eqref{eq: kostka numb} into a statement about permutations:
\begin{align*}
    K_{\shape(S),\gamma} &=  | \{w \in S_n  \ | \ P(w) = S, \des(w) \subset \{\gamma_1,\gamma_1 + \gamma_2, \dots, \gamma_1 +\cdots + \gamma_{l-1}\} \} |.
\end{align*}

Combining these observations, we have the following equality:
\begin{align}\label{eq:Hilb using words}
    \langle e_{\gamma},\  \mHL_{\mu} [X;q]\rangle  = \sum\limits_{\substack{w\in S_n\\  \ctype(P(w)^t)\unrhd\mu \\ \des(w) \subset \{\gamma_1,\gamma_1 + \gamma_2, \dots, \gamma_1 +\cdots + \gamma_{l-1}\} }}  q^{\charge(w)}.
\end{align}

From this, proving $\Frob_q(R_{\mu}) = \mHL_{\mu}[X;q]$ is equivalent to showing that for any $\gamma\models n$, the Hilbert series $\Hilb(N_\gamma R_{\mu})$ is equal to the right hand expression of \eqref{eq:Hilb using words}.
We show this by proving that the natural subset of charge monomials to expect from Equation \eqref{eq:Hilb using words} give a basis of $N_\gamma R_{\mu}$ by antisymmetrization.

Using the ungraded character of $R_\mu$, given by $h_\mu$, along with Proposition \ref{prop: dim and antisym}, we can make the following observation:
\begin{lemma}\label{lem:antisymm dimension}
    The dimension of $N_\gamma R_{\mu}$ is 
    \[|\{ w \in S_n, \ \ctype(P(w)^t) \unrhd\mu, \des(Q(w)) \subset \{\gamma_1,\gamma_1 + \gamma_2, \dots, \gamma_1 + \cdots + \gamma_{l-1}\}|\}.\]
\end{lemma}
\begin{proof}
    Since the ungraded Frobenius character of $R_{\mu}$ is $h_{\mu}$, we know that $\dim(N_\gamma R_{\mu}) = \langle e_{\gamma},\ h_{\mu} \rangle$.
    Using the fact that $\mHL_{\mu}[X;1] = h_{\mu}$, we have:
    
    \[\dim(N_\gamma R_{\mu}) = \langle e_{\gamma},\ h_{\mu} \rangle =   \langle e_{\gamma},\mHL_{\mu}[X;1] \rangle = \sum\limits_{\substack{w\in S_n\\  \ctype(P(w)^t)\unrhd\mu \\ \des(w) \subset \{\gamma_1,\gamma_1 + \gamma_2, \dots, \gamma_1 +\cdots + \gamma_{l-1}\} }}  1.\]
\end{proof}

We now use this fact to determine a basis of $N_\gamma R_{\mu}$.
\begin{prop2}
\makeatletter\def\@currentlabel{C}\makeatother
\label{prop:prop2}

    Let $\mu \vdash n$  and $\gamma = (\gamma_1,\dots, \gamma_l)\models n$. 
    The set
    \begin{align}\label{eq:antisym basis}
    \{ N_\gamma \mathbf{x}^{\cw(w)} \ | \  w \in S_n, \ \ctype(P(w)^t) \unrhd\mu, \des(Q(w)) \subset \{\gamma_1,\gamma_1 + \gamma_2, \dots, \gamma_1 + \cdots + \gamma_{l-1}\}\}.
    \end{align}
    is
    a basis of $N_\gamma R_{\mu}$.
\end{prop2}

\begin{proof}

We first show that the terms in \eqref{eq:antisym basis} are nonzero. 
Consider $w\in S_n$ such that $\ctype(\rev(w)) \unrhd \mu,$ and $\des(w) \subset \{\gamma_1,\gamma_1 + \gamma_2, \dots, \gamma_1 + \cdots + \gamma_{l-1}\}.$ 
We divide $w$ into blocks of size $\gamma_1,\gamma_2,\dots, \gamma_{l}$ where the entries within each block are strictly increasing. 
This implies that no two entries in the same block have the same charge value.
Note that $S_\gamma$ permutes indices within the blocks, thus each monomial term of $N_{\gamma}x^{\cw(w)}$ is distinct. 
Hence $N_{\gamma} x^{\cw(w)} \neq 0$.
From Lemma \ref{lem:antisymm dimension}, this shows that our set has the correct cardinality to be a basis of $N_{\gamma}R_{\mu}$. 

Now, we show that \eqref{eq:antisym basis} spans $N_\gamma R_{\mu}$. The set $\{N_\gamma \mathbf{x}^\alpha \ | \ \alpha \in \mathcal{C}_\mu\}$
spans $N_\gamma R_{\mu}$ since $\{\mathbf{x}^\alpha \ | \ \alpha \in \mathcal{C}_\mu\}$ is a basis of $R_{\mu}$.
Consider $w\in S_n$ such that $\ctype(\rev(w))\unrhd\mu$, $N_\gamma x^{\cw(w)} \neq 0$ but $w$ does not have the correct descent set. We will show that $N_\gamma\mathbf{x}^{\cw(w)} = N_\gamma\mathbf{x}^{\cw(\text{sort}_\gamma (w))}$  for some $\text{sort}_\gamma(w)\in S_n$ with the correct descent set and catabolizability type.

If we divide $\cw(w)$ into blocks of size $\gamma_1,\gamma_2,\dots, \gamma_{l}$, the entries within a block all have distinct charge labels since $N_\gamma x^{\cw(w)} \neq 0$. 
From this, we can rearrange the entries in $w$ so that they are strictly increasing within the blocks without changing the charge labels of the respective entries. That is: if $w_i > w_{i+1}$ with $\cw(w)_i \neq \cw(w)_{i+1}$, we know $w_i \neq (w_{i+1} + 1)$, hence we can swap $w_i$ and $w_{i+1}$ without changing the respective charge labels.
We can repeat this until we sort $w$ so that it is increasing within the blocks. Let $\text{sort}_\gamma(w)$ be the resulting permutation.
We see that $\des(\text{sort}_\gamma(w) )\subset \{\gamma_1,\gamma_1+\gamma_2,\dots, \gamma_1+ \cdots \gamma_{l-1}\}.$ 
 By Proposition \ref{prop: swaps raise ctype}, we know $\ctype(\rev(\text{sort}_\gamma(w))) \unrhd \ctype(\rev(w)) \unrhd \mu$, hence $N_\gamma \mathbf{x}^{\cw(\text{sort}_\gamma(w))}$ is in  \eqref{eq:antisym basis}.
\end{proof}

Using the construction of our basis, we can translate questions about the structure of $R_{\mu}$ into questions about conditions on tableaux.
In this case, choosing the elements that are a basis of $N_\gamma R_{\mu}$ is equivalent to looking at pairs of standard tableaux $(P,Q)$ of the same shape with conditions on both $P$ and $Q$. 
We illustrate this with an example:
\begin{ex}
Consider $\mu= (2,1,1)$ and $\gamma = (2,2)$. All standard tableaux $P$ that satisfy $\ctype(P^t) \unrhd \mu$ are listed below:
   \ytableausetup{mathmode,smalltableaux,  aligntableaux = bottom}
\[\ytableaushort{2,134} \ , \  \ytableaushort{24,13} \ , \ \ytableaushort{4,2,13} \ , \  \ytableaushort{3,2,14}\ , \ \ytableaushort{4,3,2,1}. \]
We also list all standard tableaux $Q$ such that $\des(Q) \subset \{2\}$:
\[\ytableaushort{1234} \ , \  \ytableaushort{3,124}, \ \ \ytableaushort{34,12}. \]
Thus the pairs $(P,Q)$ of standard tableaux of the same shape that satisfy $\ctype(P^t) \unrhd \mu$ and $\des(Q) = \{\gamma_1\}$ are the pairs $ \ytableausetup{smalltableaux,  aligntableaux = top}(\ytableaushort{2,134},\ytableaushort{3,124})$ and $(\ytableaushort{24,13},\ytableaushort{34,12})$.
These pairs corresponds to the words $2314, 2413$, which give us the charge monomials $x_2x_4^2, x_2x_4.$
From this, we know the basis of $N_\gamma R_{\mu}$ is given by the polynomials
\[\{ x_2x_4^2 - x_1x_4^2 - x_2x_3^2 + x_1x_3^2, x_2x_4 - x_1x_4 - x_2x_3 + x_1x_3 \}.\]

\end{ex}

The following corollary is immediate from Proposition \ref{prop:prop2}.

\begin{cor}
For any partition $\mu$ of $n$ we have $\Frobq(R_{\mu}) =\mHL_{\mu}[X;q].$
\end{cor}
\begin{proof}
The graded Frobenius character of a $\mathbb{C}S_n$-module $V$ is uniquely determined by the values $\langle e_\gamma , \Frob_q(V)\rangle$ for all $\gamma\vdash n$.
From Proposition \ref{prop:prop2}, we have
\[\Hilb(N_\gamma R_{\mu})  = \sum\limits_{\substack{w\in S_n\\  \ctype(P(w)^t)\unrhd\mu \\ \des(w) \subset \{\gamma_1,\gamma_1 + \gamma_2, \dots, \gamma_1 +\cdots + \gamma_{l-1}\} }}  q^{\charge(w)} \]
which is equal to $\langle e_\gamma , \Frob_q(R_{\mu})\rangle$ by Equation \eqref{eq:Hilb using words}. Thus $\Frob_q(R_{\mu}) = \mHL_{\mu}[X;q].$
\end{proof}

\begin{rem}
    Note that to avoid circular reasoning, we must ensure that we can show \eqref{eq: our basis} is a basis of $R_\mu$ without using the fact that $\Frob_q(R_\mu) =\mHL_\mu[X;q]$. Using the fact that $\mathcal{C}_\mu\subset \mathcal{D}_{\mu^t}$ and that $\mathbf{x}^{\mathcal{D}_{\mu^t}}$ is linearly independent in $R_\mu$ \cite[Corollary 3.14]{CC}, we have that \eqref{eq: our basis} is linearly independent in $R_\mu$. Furthermore, using the ungraded character, we have
    \[\dim(R_\mu) = \langle h_{1^n}, \Frob(R_\mu)\rangle = \langle h_{1^n}, h_\mu \rangle = \langle h_{1^n}, \tilde{H}_\mu[X;1] \rangle =|\mathcal{C}_\mu|. \]
    Thus we can show our set is a basis using only the ungraded character of $R_\mu$.
\end{rem}

Carlsson-Chou \cite{CC} have a similar result on the antisymmetric component with respect to a certain Young subgroup. 
We define $\sh_\gamma$ (resp. $\sh'_\gamma)$ to be the set of all permutations such that $\{1,2,\dots, \gamma_1\}, \{\gamma_1+1,\dots, \gamma_1+\gamma_2\}, \dots, \{\gamma_1+\cdots + \gamma_{l-1} + 1, \dots, n\}$ appear in increasing (resp. decreasing) order.

\begin{thm}[Theorem 3.8.2, Carlsson-Chou \cite{CC}]
    The set $\{N_\gamma g_\sigma(\mathbf{x}) \ | \  g_\sigma(\mathbf{x}) = \mathbf{x}^\alpha \text{ for }\alpha\in \mathcal{D}_\mu, \sigma \in \sh'_{\gamma}\}$ is a basis of $N_\gamma R_{\mu^t}$.
\end{thm}
Similar to what we do above, they use their result to show that $\Frobq(R_{\mu}) = \mHL_{\mu}$ by deriving a combinatorial formula for $\langle m_\gamma\rangle  \  \omega \mHL_{\mu}$ using parking functions and shuffle combinatorics which matches their description of the basis.

In contrast to their result, ours arises naturally from the catabolizability formula \eqref{eq:catabolizable HL formula} for $\mHL_{\mu}$. 
Our proof of the result is also independent of theirs: we do not relate their classification to ours when we show Proposition \ref{prop:prop2}, though they use a similar argument to show that their basis is linearly independent. 
However, it is easy to see that the two sets coincide by using $x^{\cw(w)} = g_{\rev(w^{-1})}(\mathbf{x})$ to show that the two conditions are equivalent:
\[\des(w) \subset \{\gamma_1,\gamma_1 + \gamma_2, \dots, \gamma_1 + \cdots + \gamma_{l-1}\} \Leftrightarrow w^{-1} \in \sh_\gamma  \Leftrightarrow \rev(w^{-1}) \in \sh'_\gamma.\]

\section*{acknowledgements}The author thanks Mark Haiman for helpful conversation and guidance, as well as the anonymous referees. 

\newpage
\section*{appendix}\label{sec: sppendix}
\subsection*{Example of Algorithm \ref{alg:chains}}
Let $w = 3 \ \ 9 \ \ 1 \ \ 2 \ \ 6 \ \ 7 \ \ 10 \ \ 3 \ \ 8 \ \ 4.$

Then $\cc(w) = \textcolor{blue}{\textbf{1}} \ \ \textcolor{blue}{\textbf{2}} \ \ \textcolor{blue}{\textbf{0}} \ \ \textcolor{blue}{\textbf{0}} \ \ \textcolor{red}{1} \ \ \textcolor{blue}{\textbf{1}} \ \ \textcolor{blue}{\textbf{2}} \ \ \textcolor{red}{0} \ \ \textcolor{red}{1}\ \ \textcolor{red}{0} \in \sh (\textcolor{blue}{\textbf{120012}},\textcolor{red}{1010})$, where we have colored the entries to denote how the two words are shuffled. 

The initial input when applying Algorithm \ref{alg:chains} is given by:
\ytableausetup{textmode, boxframe=normal, boxsize=2em, aligntableaux=center}
\[\left(1\ \ 2 \ \ 0 \ \ 0 \ \ 1 \ \ 1 \ \ 2 \ \ 0 \ \ 1\ \  0, \ \emptyset  \ , \ \ \ 
\begin{ytableau}
\footnotesize(3,4) \\ \footnotesize(2,9)\\ \footnotesize(2,5) \\   \footnotesize(1,2) & \footnotesize(1,6) & \footnotesize(1,10) \\\footnotesize(1,1) & \footnotesize(1,3)& \footnotesize(1,7)& \footnotesize(1,8) 
\end{ytableau} \ \ \right).\]

We can see that the first three steps of algorithm correspond to adding boxes $(1,1),(1,2),(1,3)$ to the partition in the second coordinate. We do not modify the filling in the last coordinate. The resulting triple we get is:

\begin{align}\label{appendix: ex1}
    \left(1\ \ 2 \ \ 0 \ \ 0 \ \ 1 \ \ 1 \ \ \underline{\mathbf{2}} \ \ \ \rule{.2cm}{0.15mm} \ \ \rule{.2cm}{0.15mm} \ \  \rule{.2cm}{0.15mm} \ , \ \  \ytableausetup{smalltableaux} \ydiagram{1,2}  \ , \ \ \ 
\ytableausetup{textmode, boxframe=normal, boxsize=2em, aligntableaux=center} \begin{ytableau}
\scriptsize(3,4) \\ \scriptsize(2,9)\\ \scriptsize(2,5) \\  *(lightgray) \scriptsize(1,2) & \scriptsize(1,6) & \scriptsize(1,10) \\ *(lightgray)\scriptsize(1,1) & *(lightgray)\scriptsize(1,3)& \scriptsize(1,7)& \scriptsize(1,8) 
\end{ytableau} \ \ \right).
\end{align}

The entries have been read up to this point are exactly those in the shape $\nu = \ytableausetup{smalltableaux} \ydiagram{1,2}$. These boxes are shaded.

At the fourth step (which corresponds to reading $\underline{\mathbf{2}}$), we can see that we will add a box to the partition in the second coordinate. However, we can see that $(3,4)$ is in the fifth row. Thus, we delete $(3,4)$ and insert $(1,4)$ into the third row of our filling:

\[ \ytableausetup{textmode, boxframe=normal, boxsize=2em, aligntableaux=center}\begin{ytableau}
  \none  \\ \none[(1,4) $\rightarrow$] &  \none \\  \none \\  \none 
\end{ytableau}  
\begin{ytableau}
 \scriptsize(2,9)\\ \scriptsize(2,5) \\   *(lightgray)\scriptsize(1,2) & \scriptsize(1,6) & \scriptsize(1,10) \\ *(lightgray)\scriptsize(1,1) & *(lightgray)\scriptsize(1,3)& \scriptsize(1,7)& \scriptsize(1,8) 
\end{ytableau} \Rightarrow  \begin{ytableau}
 \scriptsize(2,9)\\ \scriptsize\textbf{(1,4)}  & \scriptsize(2,5) \\   *(lightgray)\scriptsize(1,2) & \scriptsize(1,6) & \scriptsize(1,10) \\ *(lightgray)\scriptsize(1,1) &*(lightgray) \scriptsize(1,3)& \scriptsize(1,7)& \scriptsize(1,8)
\end{ytableau}.\]

The resulting triple after the fourth step is:
 
\[\left(1\ \ 2 \ \ 0 \ \ 0 \ \ 1 \ \ \underline{\textbf{1}} \ \ \rule{.2cm}{0.15mm}\ \ \ \rule{.2cm}{0.15mm} \ \ \rule{.2cm}{0.15mm} \ \  \rule{.2cm}{0.15mm} \ , \ \  \ytableausetup{smalltableaux} \ydiagram{1,1,2}  \ , \ \ \ 
\ytableausetup{textmode, boxframe=normal, boxsize=2em, aligntableaux=center} \begin{ytableau}
 \scriptsize(2,9)\\ *(lightgray)\scriptsize(1,4) &  \scriptsize(2,5) \\   *(lightgray)\scriptsize(1,2) & \scriptsize(1,6) & \scriptsize(1,10) \\*(lightgray)\scriptsize(1,1) & *(lightgray)\scriptsize(1,3)& \scriptsize(1,7)& \scriptsize(1,8)
\end{ytableau}\right).\]

At the fifth step (corresponding to reading \underline{\textbf{1}}), we add a box to the partition in the second coordinate, as well as modify the filling in the third coordinate by deleting (2,5) and  inserting (1,5) into the second row:

\[ \begin{ytableau}
  \none  \\ \none \\ \none[(1,5) $\rightarrow$] &  \none \\  \none 
\end{ytableau}  
\begin{ytableau}
 \scriptsize(2,9)\\ *(lightgray)\scriptsize(1,4)   \\   *(lightgray)\scriptsize(1,2) & \scriptsize(1,6) & \scriptsize(1,10) \\ *(lightgray)\scriptsize(1,1) & *(lightgray)\scriptsize(1,3)& \scriptsize(1,7)& \scriptsize(1,8)\end{ytableau} \Rightarrow \begin{ytableau}
 \scriptsize(2,9)\\ *(lightgray)\scriptsize(1,4)   \\   *(lightgray)\scriptsize(1,2) & \scriptsize\textbf{(1,5)} & \scriptsize(1,10)  & \none & \none[$\rightarrow$ \textbf{(1,6)}.]\\*(lightgray)\scriptsize(1,1) & *(lightgray) \scriptsize(1,3)& \scriptsize(1,7)& \scriptsize(1,8)\end{ytableau}\]

 We see that $(1,6)$ pops out of the second row. We then place $(2,6)$ into the empty spot in the third row, where $(2,5)$ was.

\[\begin{ytableau}
  \none   \\ \none[(2,6) $\rightarrow$] &  \none \\ \none \\  \none 
\end{ytableau} \begin{ytableau}
 \scriptsize(2,9)\\ *(lightgray)\scriptsize(1,4)   \\  *(lightgray) \scriptsize(1,2) & \scriptsize(1,5) & \scriptsize(1,10)  \\ *(lightgray)\scriptsize(1,1) & *(lightgray)\scriptsize(1,3)& \scriptsize(1,7)& \scriptsize(1,8)\end{ytableau} \Rightarrow \begin{ytableau}
 \scriptsize(2,9)\\*(lightgray) \scriptsize(1,4)  &\scriptsize \textbf{(2,6)}   \\ *(lightgray)  \scriptsize(1,2) & \scriptsize(1,5) & \scriptsize(1,10)  \\*(lightgray)\scriptsize(1,1) & *(lightgray)\scriptsize(1,3)& \scriptsize(1,7)& \scriptsize(1,8)\end{ytableau}.\]

 The resulting triple is 

\[\left(1\ \ 2 \ \ 0 \ \ 0 \ \  \underline{\textbf{1}} \ \  \rule{.2cm}{0.15mm} \ \ \rule{.2cm}{0.15mm}\ \ \ \rule{.2cm}{0.15mm} \ \ \rule{.2cm}{0.15mm} \ \  \rule{.2cm}{0.15mm} \ , \ \  \ytableausetup{smalltableaux} \ydiagram{1,2,2}  \ , \ \ \ 
\ytableausetup{textmode, boxframe=normal, boxsize=2em, aligntableaux=center} \begin{ytableau}
 \scriptsize(2,9)\\*(lightgray) \scriptsize(1,4)  &\scriptsize(2,6)  \\  *(lightgray) \scriptsize(1,2) &*(lightgray) \scriptsize(1,5) & \scriptsize(1,10)  \\*(lightgray)\scriptsize(1,1) &*(lightgray) \scriptsize(1,3)& \scriptsize(1,7)& \scriptsize(1,8)\end{ytableau}
\right).\]

At the next step, we cannot add anything to the second row. So we add one to the element in the word without changing either of the shapes:

\[\left(1\ \ 2 \ \ 0 \ \ \textbf{\underline{0}} \ \  2  \ \  \rule{.2cm}{0.15mm} \ \ \rule{.2cm}{0.15mm}\ \ \ \rule{.2cm}{0.15mm} \ \ \rule{.2cm}{0.15mm} \ \  \rule{.2cm}{0.15mm} \ , \ \  \ytableausetup{smalltableaux} \ydiagram{1,2,2}  \ , \ \ \ 
\ytableausetup{textmode, boxframe=normal, boxsize=2em, aligntableaux=center} \begin{ytableau}
 \scriptsize(2,9)\\ *(lightgray)\scriptsize(1,4)  &\scriptsize(2,6)  \\  *(lightgray) \scriptsize(1,2) & *(lightgray)\scriptsize(1,5) & \scriptsize(1,10)  \\ *(lightgray)\scriptsize(1,1) & *(lightgray)\scriptsize(1,3)& \scriptsize(1,7)& \scriptsize(1,8)\end{ytableau}
\right).\]

The next two steps just add boxes to the first row of the partition in the second coordinate without changing the filling. We get
\[\left(1\ \ \textbf{\underline{2}} \ \ \rule{.2cm}{0.15mm} \ \ \rule{.2cm}{0.15mm} \ \  2  \ \  \rule{.2cm}{0.15mm} \ \ \rule{.2cm}{0.15mm}\ \ \ \rule{.2cm}{0.15mm} \ \ \rule{.2cm}{0.15mm} \ \  \rule{.2cm}{0.15mm} \ , \ \  \ytableausetup{smalltableaux} \ydiagram{1,2,4}  \ , \ \ \ 
\ytableausetup{textmode, boxframe=normal, boxsize=2em, aligntableaux=center} \begin{ytableau}
 \scriptsize(2,9)\\ *(lightgray)\scriptsize(1,4)  &\scriptsize(2,6)  \\   *(lightgray)\scriptsize(1,2) & *(lightgray)\scriptsize(1,5) & \scriptsize(1,10)  \\*(lightgray)\scriptsize(1,1) &*(lightgray) \scriptsize(1,3)& *(lightgray)\scriptsize(1,7)& *(lightgray)\scriptsize(1,8)\end{ytableau}
\right).\]

At the next step, we delete $(2,9)$ and insert (1,9) into the third row of the filling to get:

\[\left(1\ \ \rule{.2cm}{0.15mm} \ \ \rule{.2cm}{0.15mm} \ \ \rule{.2cm}{0.15mm} \ \  2  \ \  \rule{.2cm}{0.15mm} \ \ \rule{.2cm}{0.15mm}\ \ \ \rule{.2cm}{0.15mm} \ \ \rule{.2cm}{0.15mm} \ \  \rule{.2cm}{0.15mm} \ , \ \  \ytableausetup{smalltableaux} \ydiagram{2,2,4}  \ , \ \ \ 
\ytableausetup{textmode, boxframe=normal, boxsize=2em, aligntableaux=center} \begin{ytableau}
*(lightgray) \scriptsize(1,4) & *(lightgray)\scriptsize(1,9)  &\scriptsize(2,6)  \\ *(lightgray)  \scriptsize(1,2) & *(lightgray)\scriptsize(1,5) & \scriptsize(1,10)  \\*(lightgray)\scriptsize(1,1) & *(lightgray) \scriptsize(1,3)& *(lightgray)\scriptsize(1,7)& *(lightgray)\scriptsize(1,8)\end{ytableau}
\right).\]

Now, we can continue to cycle through the word and add boxes to the  partition until we get 

\[\left(\rule{.2cm}{0.15mm}\ \ \rule{.2cm}{0.15mm} \ \ \rule{.2cm}{0.15mm} \ \ \rule{.2cm}{0.15mm} \ \  \rule{.2cm}{0.15mm}  \ \  \rule{.2cm}{0.15mm} \ \ \rule{.2cm}{0.15mm}\ \ \ \rule{.2cm}{0.15mm} \ \ \rule{.2cm}{0.15mm} \ \  \rule{.2cm}{0.15mm} \ , \ \  \ytableausetup{smalltableaux} \ydiagram{3,3,4}  \ , \ \ \ 
\ytableausetup{textmode, boxframe=normal, boxsize=2em, aligntableaux=center} \begin{ytableau}
*(lightgray) \scriptsize(1,4) & *(lightgray)\scriptsize(1,9)  &*(lightgray)\scriptsize(2,6)  \\  *(lightgray) \scriptsize(1,2) & *(lightgray)\scriptsize(1,5) & *(lightgray)\scriptsize(1,10)  \\*(lightgray)\scriptsize(1,1) &*(lightgray) \scriptsize(1,3)& *(lightgray)\scriptsize(1,7)& *(lightgray)\scriptsize(1,8)\end{ytableau}
\right).\]

\newpage 
\bibliographystyle{amsplain}
\bibliography{references}

\end{document}